\documentclass[11pt,a4paper]{article}

\title{\bf Rank Contributions of Vertices in Rigidity Matroids of Clique Covered Graphs}

\author{
Bill Jackson\thanks{Queen Mary University of London, London, E1 4NS, UK, and
the HUN-REN-ELTE Egerv\'ary Research Group
on Combinatorial Optimization, P\'azm\'any P\'eter s\'et\'any 1/C, 1117 Budapest, Hungary.
e-mail: \texttt{b.jackson@qmul.ac.uk}}
\and
Tibor Jord\'an\thanks{Department of Operations Research, ELTE E\"otv\"os Lor\'and University, and the HUN-REN-ELTE Egerv\'ary Research Group
on Combinatorial Optimization, P\'azm\'any P\'eter s\'et\'any 1/C, 1117 Budapest, Hungary.
e-mail: \texttt{tibor.jordan@ttk.elte.hu}}
\and 
Soma Vill\'anyi\thanks{Department of Operations Research, ELTE E\"otv\"os Lor\'and University, and the HUN-REN-ELTE Egerv\'ary Research Group
on Combinatorial Optimization, P\'azm\'any P\'eter s\'et\'any 1/C, 1117 Budapest, Hungary. e-mail: \texttt{soma.villanyi@ttk.elte.hu}}}

\usepackage[utf8]{inputenc}
\usepackage[english]{babel}
\usepackage{amsmath}
\usepackage{amsthm}
\usepackage{amssymb}
\usepackage{geometry}
\usepackage{a4wide}
\usepackage{bm}
\usepackage{microtype}
\usepackage{mathtools}
\usepackage{csquotes}
\usepackage[shortlabels]{enumitem}
\usepackage[nottoc,numbib]{tocbibind}
\usepackage{changepage}
\usepackage{cleveref}
\usepackage{enumitem}
\usepackage{xcolor}
\setlist[enumerate]{itemsep=0mm,parsep=2mm}

\theoremstyle{definition}
\newtheorem{theorem}{Theorem}[section]

\newtheorem{lemma}[theorem]{Lemma}
\newtheorem*{lemma*}{Lemma}
\newtheorem*{conjecture*}{Conjecture}
\newtheorem*{lemma''*}{``Lemma''}
\newtheorem{claim}[theorem]{Claim}
\newtheorem*{claim*}{Claim}
\newtheorem{corollary}[theorem]{Corollary}

\newtheorem{conjecture}[theorem]{Conjecture}

\newtheorem{example}[theorem]{Example}
\newtheorem{proposition}[theorem]{Lemma}

\DeclareMathOperator{\rr}{r}
\DeclareMathOperator{\rc}{rc}
\DeclareMathOperator{\cl}{cl}

\DeclareMathOperator{\rank}{rank}

\DeclareMathOperator{\val}{val}
\DeclareMathOperator{\cval}{cval}

\newcommand{\cR}{\mathcal{R}}

\newcommand{\cX}{\mathcal{X}}
\newcommand{\cM}{\mathcal{M}}
\newcommand{\cH}{\mathcal{H}}

\newcommand{\dof}{{\rm dof}}

\newcommand{\R}{\mathbb{R}}

\newcommand{\XX}{\mathcal{X}}
\newcommand{\hXX}{\hat{\mathcal{X}}}

\newcommand{\cC}{\mathcal{C}}
\newcommand{\cP}{\mathcal{P}}

\newcommand{\HH}{\mathcal{H}}

\newcommand{\E}{\mathbb{E}}
\newcommand{\Prob}{\mathbb{P}}

\date{28 July, 2026}

\begin{document}
\maketitle
\begin{abstract}
The problems of characterizing the graphs $G$ which are generically rigid in $\R^d$, or more generally, determining the rank function of the $d$-dimensional rigidity matroid $\cR_d(G)$ of an arbitrary  graph $G$, have been solved when $d\leq 2$ but  are major open problems in discrete geometry when $d\geq 3$. In this paper we shall concentrate on the case when $d=3$. 
We first revisit a conjecture of Dress from 1987 that the rank of the ${\cal R}_3$-closure of a graph $G$ is determined by its maximal complete subgraphs of size at least five.
We show that his conjectured value for the rank of the closure gives an upper bound on the actual value. We also deduce
that the truth of this conjecture would imply a good characterization of the rank of ${\cal R}_3(G)$ for all graphs $G$. The rank formula in Dress's conjecture leads us to consider 
the 
family of $K_t$-covered graphs, i.e., graphs in which every edge belongs to a complete subgraph $K_t$, for some $t\geq 3$.
This family 
contains several well-studied graph classes such as body-pin graphs, combinatorial zeolites, and molecular graphs.

We introduce a new notion of 
rank contributions of vertices in an arbitrary matroid on the edge set of a graph $G$, and use it to obtain lower bounds on the rank contributions of vertices in $\cR_3(G)$ and $\cC^1_2(G)$ when $G$ is $K_t$-covered. We use these bounds to show that a conjectured min-max formula for the rank of body-pin graphs in ${\cal R}_3$ holds for the $C_2^1$-cofactor matroid  (which is
conjectured by Whiteley to 
be equal to 
${\cal R}_3$), 
and to 
obtain new sufficient connectivity conditions for the (global) rigidity of $K_4$- and $K_5$-covered graphs in $\R^3$.
\smallskip


\end{abstract}

\section{Introduction}

The problems of characterizing the graphs $G$ which are generically rigid in $\R^d$, or more generally, determining the rank function $r_d$ of the $d$-dimensional rigidity matroid $\cR_d(G)$ of an arbitrary  graph $G$, have been solved when $d\leq 2$ but  are major open problems in discrete geometry when $d\geq 3$. 
The rank function of $\cR_2$ was characterised
by Lov\'asz and Yemini \cite{LY} in terms of edge partitions. Their result implies the following:

\begin{theorem} 
\cite{LY}
\label{thm:lamanrank1} Let $G$ be an $\mathcal{R}_2$-closed graph,
$\mathcal{X}$ be the set of maximal cliques in $G$ of size at least four and $F$ be the set of edges of $G$ which are not covered by $\XX$. 
Then $$\rr_2(G)=|F|+\sum_{X\in \mathcal{X}}(2|X|-3).$$
\end{theorem}

In this paper we shall concentrate on the 3-dimensional rigidity matroid $\cR_3$. 
Our starting point is the following conjecture made by Andreas Dress 
at a conference in Montreal 
in 1987, see \cite{GSS} and also \cite{CDT,Taycon}, which proposes a similar expression for the rank of  an $\cR_3$-closed graph, i.e., a graph $G\subseteq K_n$  whose edge set is a closed set in $\cR_3(K_n)$. 
Given  a collection $\cal X$  of subsets of a ground set $V$ and  $h=\{u,v\}\subseteq  V$, we use
$\deg_{\cal X}(h)$ to denote the number of sets $X\in {\cal X}$ with $u,v\in X$. %
In addition, we say that $\{u,v\}$ is a {\em hinge} of $\cX$ if $\deg_{\cal X}(h)\geq 2$, and denote the set of all hinges of $\cX$ by $\cH(\cX)$.

\begin{conjecture}[The Dress Conjecture]\label{con:dress}
Let $G$ be an $\cR_3$-closed graph, $\hXX$ be the set of maximal cliques in $G$ of size at least five and $F$ be the set of edges of $G$ which are not covered by $\hXX$. Then 
\begin{equation}\label{eq:dress}
r_3(G)=|F|+\sum_{X\in \hXX}(3|X|-6)-\sum_{h\in \mathcal{H(\hXX)}}(\deg_{\hXX}(h)-1).
\end{equation}
\end{conjecture}
\noindent
It is not obvious that this conjecture would lead to a good characterisation of $r_3$. Another conjecture of Dress et al.\ \cite{Dress1} which would have given a good characterisation was shown to be false in \cite{JJdress}.

Clinch, Jackson and Tanigawa \cite{CJT} recently provided substantial evidence in favour of Conjecture \ref{con:dress} by showing that the analogous conjecture  holds when we replace the 3-dimensional rigidity matroid $\cR_3$ by the cofactor matroid $C_2^1$ (which we will define in Section \ref{sec:cofdefn}).

\begin{theorem}\label{thm:cjt}
\cite[Theorem 6.3]{CJT}
Let $G$ be a $\cC_2^1$-closed graph, $\hXX$ be the set of maximal cliques in $G$ of size at least five and $F$ be the set of edges of $G$ which are not covered by $\hXX$. Then the rank of $\cC_2^1(G)$ is given by
$$r_2^1(G)=|F|+\sum_{X\in \hXX}(3|X|-6)-\sum_{h\in \cH(\hXX)}(\deg_{\hXX}(h)-1).$$
\end{theorem}

Conjecture \ref{con:dress} would follow from Theorem \ref{thm:cjt} if the following long-standing conjecture of  Whiteley \cite{Wlong} is valid.

\begin{conjecture}\label{con:walter}
$\cR_3(K_n)=\cC_2^1(K_n)$ for all $n\geq 2$.
\end{conjecture}

We will show in Section \ref{sec:DW} that
the right hand side of (\ref{eq:dress}) 
gives an upper bound 
on  
$r_3(G)$ when $G$ is $\cR_3$-closed, and that
Conjectures \ref{con:dress} and \ref{con:walter} are in fact equivalent. Hence Conjecture \ref{con:dress} would imply that the NP$\cap$CoNP characterization of the rank function of $\cC_2^1(K_n)$ given in \cite{CJT} also holds for $\cR_3(K_n)$.

The Dress Conjecture motivates our primary concern in this paper: the 3-dimensional rigidity and $C_2^1$-cofactor matroids of  {\em $K_t$-covered graphs}, i.e., graphs with the property that each of their edges  belongs to a copy of $K_t$. A closely related motivation for our study is the result of \cite{CJT} that every 
graph whose edge set is a bridgeless closed set in the $C_2^1$-cofactor matroid is
$K_5$-covered. In addition, the family of $K_4$-covered graphs  
includes several important graph classes such as 
body-pin graphs ($K_4$-covered graphs in which every vertex belongs to at most two copies of $K_4$), 
3-dimensional combinatorial zeolites (line graphs of 4-regular graphs), and molecular graphs (squares of graphs of minimum degree at least $3$).

We will show that a conjectured expression for the rank function of the 3-dimensional rigidity matroid of a body-pin graph determines its rank in the $C_2^1$-cofactor matroid (Theorem \ref{thm:body-pin_C12}), and obtain several new sufficient connectivity conditions for the rigidity and global rigidity of
$K_t$-covered graphs in $\R^3$ when $t=4,5$.
Our results refine  previous  sufficient connectivity conditions \cite{LY,JJconnrig}
for an arbitrary graph to be rigid or globally rigid in $\R^3$ and
provide supporting evidence for several other conjectures on 3-dimensional rigidity. More precisely, we show that: every 5-connected $K_5$-covered graph is globally rigid in $\R^3$ (Theorem \ref{thm:5c5c:GR}); every  $K_4$-covered graph is $C^1_2$-rigid if it is 6-connected (Theorem \ref{thm:zeo:general}), and is rigid in  $\R^3$ if it is either  7-connected (Theorem \ref{thm:7c4c_rigid}) or  5-connected and $\cR_3$-bridgeless (Theorem \ref{thm:K4cyclic}).

We prove our results on
$K_t$-covered graphs by considering the `rank contributions' of the vertices in a matroid defined on the edge set of a graph. This concept was implicitly
used by Lov\'asz and Yemini \cite{LY}
to show that 6-connected graphs are rigid in $\R^2$. The power of this approach, when combined with the probabilistic method, became apparent in  a recent paper of the third author \cite{vill}, where it was used to obtain the best possible sufficient connectivity condition for a graph to be (globally) rigid in $\R^d$, confirming
a long-standing conjecture from \cite{LY}. We will develop a general 
theory of 
rank contributions of vertices in an arbitrary matroid on the edge set of a graph $G$, and then apply it to $\cR_2(G)$, $\cR_3(G)$ and $\cC^1_2(G)$. Further applications of rank contributions to $\cR_d(G)$ are given in \cite{JLV}.

The structure of the paper is as follows. 
Section 2 contains preliminary results on rigid graphs and the rigidity matroid.
The conjectures of Dress and Whiteley are discussed in Section 3. 
Rank contributions of vertices in matroids on the edge set of a complete graph
are introduced in Section 4. The cases when the matroid is equal to 
$\cR_2$, or  $\cR_3$ or $C_2^1$, are investigated in Sections 5 and 6, respectively.
Our results on the
$C^1_2$-rigidity of $K_4$-covered graphs and body-pin frameworks are in Section 7,
while the results on
the rigidity and global rigidity of 
$K_t$-covered graphs in $\R^3$
are in Section 8. We close by stating two conjectures on the rigidity of $K_4$-covered graphs in $\R^3$ in Section 9.

\section{Preliminaries}

\subsection{The $d$-dimensional rigidity matroid}

Let $d$ be a positive integer. Given a graph  $G=(V,E)$ and a map $p:V\to \R^d$, the {\em rigidity matrix} of the {\em framework} $(G,p)$ is the $|E|\times d|V|$ matrix $R(G,p)$ in which the row indexed by an edge $uv\in E$ has $p(u)-p(v)$ in the $d$ columns indexed by $u$, $p(v)-p(u)$ in the $d$ columns indexed by $v$, and zeros elsewhere.
The {\it $d$-dimensional rigidity matroid} of the graph $G=(V,E)$, denoted by ${\cal R}_d(G)$,
is  the row matroid of $R(G,p)$ for any generic $p:V\to \R^d$.
This matroid encodes several fundamental rigidity properties of 
generic realisations of 
$G$ in $\R^d$. 
For example, the rank of ${\cal R}_d(G)$, denoted by $\rr_d(G)$, satisfies
\begin{equation}
\label{gl}
\rr_d(G)\leq d|V|-\binom{d+1}{2} \mbox{ whenever $|V|\geq d+1$,}
\end{equation}
and equality holds if and only if every generic realisation $p$ of $G$ in $\R^d$ is rigid, i.e., every continuous motion of the vertices of $(G,p)$ in $\R^d$ which preserves the lengths of its edges results in a congruent realisation.
We say that $G$ is {\it rigid in $\R^d$} (or simply {\it $d$-rigid}) if either $|V|\leq d$ and
$G$ is complete, or $|V|\geq d+1$ and equality holds in (\ref{gl}).
The number of
{\it degrees of freedom} of $G$ in $\R^d$, denoted by $\dof_d(G)$, is defined to be the
difference between $\rr_d(K_{|V|})$ and $\rr_d(G)$.
Thus, if $|V|\geq d+1$, then $\dof_d(G)=d|V|-\binom{d+1}{2}-\rr_d(G)$.
A non-adjacent pair $u,v\in V$ is {\it ${\cal R}_d$-linked} (resp.
{\it ${\cal R}_d$-loose}) in $G$ 
if $\rr_d(G+uv)=\rr_d(G)$ (resp. $\rr_d(G+uv)=\rr_d(G)+1$) holds.
The {\it ${\cal R}_d$-closure} of $G$, denoted by $\cl_d(G)$, is 
obtained from $G$
by adding an edge $uv$ between all non-adjacent ${\cal R}_d$-linked pairs $u,v$ of $G$.
The graph $G$ is called
{\it ${\cal R}_d$-closed} if $\cl_d(G)=G$ holds.

We call an edge $e\in E$ an {\itshape ${\cal R}_d$-bridge} of $G$
if $\rr_d(G-e)=\rr_d(G)-1$. An edge which is not an ${\cal R}_d$-bridge is said to be {\em ${\cal R}_d$-redundant}. We say that $G$ is 
{\itshape ${\cal R}_d$-independent} (resp. 
an 
{\itshape ${\cal R}_d$-circuit}), 
if $E$ is independent (resp.
$E$ is a circuit) in ${\cal R}_d(G)$. Thus an edge is
${\cal R}_d$-redundant in $G$ if and only if it is  contained
in an ${\cal R}_d$-circuit in $G$. Note that (\ref{gl}) tells us that a necessary condition for $G$ to be ${\cal R}_d$-independent is that 
\begin{equation}
\label{ind}
|E(H)\leq d|V(H)|-\binom{d+1}{2} \mbox{ for all $H\subseteq G$ with $|V(H)|\geq d+1$.}
\end{equation}

We say that $G$ is {\itshape ${\cal R}_d$-connected} if $\cR_d(G)$ is a connected matroid, i.e., every pair of edges in $E$ belong to an $\cR_d$-circuit in $G$. 
A $d$-dimensional framework $(G,p)$ is called {\it globally rigid} in $\R^d$
if every realisation $(G,q)$ of $G$ in $\R^d$ which has the same edge lengths as $(G,p)$ is congruent to $(G,p)$.
It is known that the global rigidity of a generic framework $(G,p)$ depends only on $G$. Thus we call a graph $G$ {\it globally rigid in $\R^d$} if every, or equivalently,
if some generic framework $(G,p)$ in $\R^d$ is globally rigid.
Garamv\"olgyi, Gortler and Jord\'an \cite{GGJ} have recently shown that the 
${\cal R}_d$-connectivity is a necessary condition for 
the global rigidity of a graph on at least $d+2$ vertices.
We will need the following result which shows that $\cR_d$-connectedness is a hereditary property.

\begin{theorem}\cite[Theorem 5.1]{GGJ}
\label{dimdrop}
Suppose that $G$ is ${\cal R}_{d+1}$-connected. Then $G$ is 
${\cal R}_{d}$-connected.
\end{theorem}

\subsection{Graph operations that preserve rigidity}\label{sec:op}

Let $G =(V,E)$ be a graph. The $d$-dimensional $0$-$extension$ operation adds a new vertex $v$ to $G$ and $d$ new edges incident with $v$. The $1$-$extension$ operation removes an edge $uw$, and adds a new vertex $v$ and $d+1$ new edges, including $vu$, $vw$.
The following lemmas are well-known, for a proof, see e.g.\ \cite{Wlong}.

\begin{lemma}
\label{lem:1ext}
Let $G,H$ be  graphs.\\
(a) If $H$ is rigid in $\R^d$ and $G$ is obtained from $H$ by a $0$- or
$1$-extension then $G$ is rigid in $\R^d$.
(b) If $G$ is obtained from $H$ by a $0$-extension operation which adds a new vertex $v$ adjacent to $d$ vertices  $v_1,v_2,\ldots,v_d$ of $H$ and $p$ is a realisation of $G$ in $\R^d$ such that $p(v),p(v_1),\ldots,p(v_d)$ are affinely independent, then $\rank R(G,p)=\rank R(H,p|_H)+d$.
\end{lemma}

\begin{lemma} (Gluing lemma)
\label{gluing}
Let $G_1=(V_1,E_1)$ and $G_2=(V_2,E_2)$ be $d$-rigid graphs with
$|V_1\cap V_2|\geq d$. Then $G_1\cup G_2$ is $d$-rigid.
\end{lemma}

Given a graph $G$, we define the {\em cone graph} of $G$ to be the graph $H$ obtained from $G$ by adding 
a new vertex $w$ to $G$ and joining $w$ to every vertex of $G$. 
We will denote $H$ by $G*w$ and refer to $w$ as the {\em cone vertex} of $G*w$. 
The following theorem is a fundamental result of Whiteley \cite{Whcone,Wlong}. %

\begin{theorem} 
\label{thm:coning} 
Let $G*w$ be the cone graph of a graph $G$. Then  
$G$ is  $d$-rigid if and only if $G*w$ is $(d+1)$-rigid.
\end{theorem}

\subsection{The 2-dimensional rigidity matroid}

It is not difficult to check that condition (\ref{ind}) is both necessary and sufficient for independence in $\cR_1$. Pollaczek-Geiringer \cite{PG},
and independently Laman \cite{laman}
showed that it also characterises independence in $\cR_2$.

\begin{theorem} \cite{laman,PG} 
\label{hilda} A graph $G$ is ${\cal R}_2$-independent if and only if 
$|E(H)|\leq 2|V(H)|-3$ for all $H\subseteq G$ with $|V(H)|\geq 2$.
\end{theorem}

This result implies that ${\cal R}_2$-circuits and ${\cal R}_2$-connected graphs are 2-rigid, and that the maximal cliques in the
$\cR_2$-closure of a graph $G$
correspond to the maximal $\cR_2$-rigid subgraphs of $G$.
Lov\'asz and Yemini \cite{LY} showed that Theorem \ref{hilda} gives rise to a characterisation of $r_2(G)$. 
We shall use the following version of their result (see \cite{Jmemoirs} for more details).
A collection ${\cal X}=\{X_1,X_2,\dots,X_t\}$ of vertex sets of size at least two of a graph $G=(V,E)$
is called a {\it cover} of $G$ if each edge in $E$ is induced by at least one
member of $\cal X$ in $G$.

\begin{theorem} \cite{LY}
\label{thm:lamanrank} 
Let $G=(V,E)$ be a graph. Then 
\begin{enumerate}
\item[(a)] $r_2(G) = \min \{ \sum_{X_i\in {\cal X}} (2|X_i|- 3)\ : {\cal X} \ \hbox{is a cover of}\  G\}.$
\item[(b)] If $G$ is  $\mathcal{R}_2$-closed  and $\cX$ consists of the vertex sets of the maximal cliques of $G$ of size at least two, then $\rr_2(G)=\sum_{X_i\in {\mathcal X}} (2|X_i|-3).$
\end{enumerate}
\end{theorem}

\subsection{Cofactor matroids}\label{sec:cofdefn}

For each integer $d\geq 1$, let $h_d:\R^2\to \R^d$ be defined by putting
$$h_d(x,y)=(x^{d-1},x^{d-2}y,\ldots,y^{d-1}).$$ Then the {\it $C^{d-2}_{d-1}$-cofactor matroid} of a graph $G=(V,E)$, denoted by $\cC^{d-2}_{d-1}(G)$,
is defined by first choosing a generic $p:V\to \R^2$ and a reference ordering for the vertices of $G$, and then taking $\cC^{d-2}_{d-1}(G)$ to be the row matroid of the $|E|\times d|V|$ matrix in which the row indexed by an edge $uv\in E$ with $u<v$ has $h_d(p(u)-p(v))$ in the $d$ columns indexed by $u$, $-h_d(p(u)-p(v))$ in the $d$ columns indexed by $v$, and zeros elsewhere. It is known that the rank of $\cC^{d-2}_{d-1}(G)$, denoted by $\rr^{d-2}_{d-1}(G)$, satisfies a similar  inequality to (\ref{gl}): 

\begin{equation}
\label{cl}
\rr^{d-2}_{d-1}(G)\leq d|V|-\binom{d+1}{2} \mbox{ whenever $|V|\geq d+1$.}
\end{equation}
By analogy, we say that $G$ is {\em $\cC^{d-2}_{d-1}$-rigid} if either equality holds in (\ref{cl}), or  $G$ is a complete graph when $|V|\leq d$.

We will need the following results of Whiteley which show that the fundamental properties of rigidity matroids given in Section \ref{sec:op} also hold for cofactor matroids (see \cite[Sections 10.1 and 10.2]{Wlong} for $d=3$ and \cite[Section 11.3]{Wlong} for general $d$).

\begin{lemma}
\label{lem:c1ext}
Let $G,H$ be  graphs. If $H$ is $\cC^{d-2}_{d-1}$-rigid  and $G$ is obtained from $H$ by a $0$- or
$1$-extension then $G$ is $\cC^{d-2}_{d-1}$-rigid.
\end{lemma}

\begin{lemma} (Cofactor Gluing Lemma)
\label{cgluing}
Let $G_1=(V_1,E_1)$ and $G_2=(V_2,E_2)$ be $\cC^{d-2}_{d-1}$-rigid graphs with
$|V_1\cap V_2|\geq d$. Then $G_1\cup G_2$ is $\cC^{d-2}_{d-1}$-rigid.
\end{lemma}

\begin{theorem} 
\label{thm:coning_cof} 
Let $G*w$ be the cone graph of a graph $G$. Then  
$G$ is  $\cC^{d-2}_{d-1}$-rigid if and only if $G*w$ is $\cC^{d-1}_{d}$-rigid.
\end{theorem}

We will also need the following result on $\cC^1_2$-cofactor matroids which was a key result in the characterisation of these matroids given in \cite{CJT}.

\begin{theorem}\cite{CJT}\label{thm:CJT:K5}
Let $G=(V,E)$ be a $\mathcal{C}_2^1$-closed graph and $e\in E$. Then $e$ is a $\mathcal{C}_2^1$-bridge in $G$ if and only if $e$ is not contained in a copy of $K_5$ in $G$.
\end{theorem}

It is known that $\cC^{d-2}_{d-1}(K_n)=\cR_{d}(K_n)$ when $d\in \{1,2\}$, and that 
$\cC^{d-2}_{d-1}(K_n)\neq \cR_{d}(K_n)$ when $d\geq 4$ and $n\geq 2d+2$. Conjecture \ref{con:walter} states that the two matroids are identical also when $d=3$. We refer the reader to \cite{Wlong} for more information on cofactor matroids.

\subsection{Operations which preserve $\mathcal{R}_3$-independence}
\label{sec:operations}

Our first result describes a well-known special case in which the so called `X-replacement operation' preserves independence in the 3-dimensional rigidity matroid. 

\begin{lemma}\label{lem:coplanar}
Let $G=(V,E)$ be a graph and $e=x_ey_e,f=x_fy_f$ be two non-adjacent edges of $G$ which are bridges in $\mathcal{R}_3(G)$. Suppose that $x_e$ and $y_e$ have a common neighbour $z$ with $x_f\neq z\neq y_f$. Let $G'$ be the graph obtained from $G-e-f$ by adding a new vertex $v$ and connecting $v$ to $x_e,y_e,x_f,y_f, z$.
Then $\rr_3(G')=\rr_3(G)+3$.
\end{lemma}
\begin{proof}
Clearly, $\rr_3(G')\leq \rr_3(G-e-f)+5= \rr_3(G)+3$. To prove $\rr_3(G')\geq \rr_3(G)+3$, let $(G,p)$ be a generic realization of $G$. Let $A$ be the plane of $p(x_e),p(y_e)$ and $p(z)$, $B$ the line through $p(x_f)$ and $p(y_f)$, and $Q$ the point where $A$ and $B$ intersect. Let $p':V\cup \{v\}\to \R^3$ be an extension of $p$ with $p'(v)=Q$.  We will show that $\rank R(G',p')\geq \rr_3(G)+3$. Let $G''=G'+e+f$. Since $p(x_e),p(x_f),p(z)$ and $Q$ are affinely independent, we have $\rr_3(G'', p')\geq \rr_3(G)+3$ by Lemma \ref{lem:1ext}(b).
It follows from the coplanarity of $p(x_e),p(y_e)$, $p(z)$ and $Q$ that $e$ is not a bridge  in the row matroid of $R(G'', p')$. Moreover, the collinearity of $p(x_f)$, $p(y_f)$ and $Q$ implies that $f$ 
is not a bridge 
in the row matroid of $R(G''-e, p')$. Thus  \begin{equation*}
\rr_3(G')\geq \rank R(G',p')= \rank R(G'',p') \geq \rr_3(G)+3.
\qedhere
\end{equation*}
\end{proof}

Our next lemma describes a  special case in which the so called `V-replacement operation' preserves independence in the 3-dimensional rigidity matroid. 
Suppose $z$ is a vertex of a graph $H$. 
The {\em $d$-dimensional vertex splitting operation} constructs a new graph $G$ from $H$  by deleting $z$ and then adding two new vertices $u$ and $v$ with $N_G(u)\cup N_G(v)=N_H(z)\cup \{u,v\}$ 
and $|N_G(u)\cap N_G(v)|= d-1$.
We can view vertex splitting as an inverse operation to contracting the edge $uv$.
Whiteley \cite{Wvsplit,Wlong} showed that 
this operation preserves ${\cal R}_d$-independence.

\begin{lemma}\label{lem:2ext:split}
Let $G=(V,E)$ be a graph, $y\in V$ and $x,z,y_1,y_2$ be distinct neighbours of $y$ in $G$. 
Suppose that 
$e=xy,f=zy$ 
are both bridges in $\mathcal{R}_3(G)$. 
Let $G'$ be the graph obtained from $G-e-f$ by adding a new vertex $v$ and connecting $v$ to $x,y,y_1,y_2, z$.
Then $\rr_3(G')=\rr_3(G)+3$.
\end{lemma}
\begin{proof}
Clearly, $\rr_3(G')\leq \rr_3(G-e-f)+5= \rr_3(G)+3$. To prove $\rr_3(G')\geq \rr_3(G)+3$, let
$G_0$ be a maximal $\mathcal{R}_3$-independent subgraph of $G$ such that  
$e,f,yy_1,yy_2$ are edges of $G_0$. Consider the graph $G_0'=(G_0-e-f)+v+vy+vy_1+vy_2+vx+vz$. Then $G_0'$ can be obtained from $G_0$ by a 3-dimensional vertex splitting operation.
Since vertex splitting preserves $\mathcal{R}_3$-independence, 
$G_0'$ is $\mathcal{R}_3$-independent. Since $G_0'$ is a subgraph of $G'$, we have $\rr_3(G')\geq \rr_3(G_0') = \rr_3(G_0)+3=\rr_3(G)+3$.
\end{proof}

\section{The conjectures of Dress and Whiteley}\label{sec:dress}
\label{sec:DW}

We will show that
the right hand side of (\ref{eq:dress}) 
gives an upper bound 
on  
$r_3(G)$, and then use this to deduce that
Conjectures \ref{con:dress} and \ref{con:walter} are equivalent.

For an arbitrary  graph $G=(V,E)$, we define:
$$F=F(G):=\{e\in E: e \text{ is not contained in a copy of } K_5 \text{ in }G\},$$
$$\hXX=\hat{\mathcal{X}}(G):=\{V(C): C \text{ is a maximal clique of $G$ with $|V(C)|\geq 5$}\}.$$
For an edge $uv\in E$, let $\deg_{\hXX}(uv)=\big|\{X\in \hXX:\{u,v\}\subseteq X\}\big|$. An edge $h\in E$ is called a {\it hinge} if  $\deg_{\hXX}(h)\geq 2$. Let
$$\mathcal{H}=\mathcal{H}(G):=\{uv\in E: uv \text{ is a hinge edge of } G\}.$$
For $v\in V$, let $E_v$ (resp.\ $F_v$, $\hXX_v$, $\mathcal{H}_v)$ denote the set of those elements of $E$ (resp.\ $F$, $\hXX, \mathcal{H}$) that are incident with $v$. 
We use the function given on the right-hand side of the equalities in Conjecture \ref{con:dress} and Theorem \ref{thm:cjt} to define the {\em clique value} of the graph $G$: 
\begin{align}
\cval(G):&
=|F|+\sum_{X\in \hXX}(3|X|-6)-\sum_{h\in \mathcal{H}}(\deg_{\hXX}(h)-1)\\
&=\sum_{X\in \hXX}(3|X|-6)-\sum_{e\in E}(\deg_{\hXX}(e)-1),
\end{align}
where the last equality uses the identity 
$$\sum_{e\in E}(\deg_{\hXX}(e)-1)=\sum_{h\in \mathcal{H}}(\deg_{\hXX}(h)-1)-|F|.$$

We first verify a useful recursive property of the function $\cval$. 

\begin{lemma}\label{lem:circuit:NGv}
Let $G$ be a graph and $v\in V$. Suppose that any two distinct maximal cliques of $G$ intersect in at most 3 vertices. Then: 
\begin{enumerate}[(a)]
\item $\cval(G)-\cval(G-v)= 3|\hXX_v|-\sum_{e\in E_v}(\deg_{\hXX}(e)-1)$; 

\item  $\cval(G)-\cval(G-v)=\cval(G[N_G(v)\cup \{v\}])-\cval(G[N_G(v)])$.
\end{enumerate}
\end{lemma}
\begin{proof} To verify (a), let $\mathcal{X}_0= \hXX(G-v)$ and $\mathcal{X}_1=\{X-\{v\}:X\in \hXX(G)\}$.
Then $\mathcal{X}_1\supseteq \mathcal{X}_0$. Furthermore, every set $X\in \mathcal{X}_1- \mathcal{X}_0$ induces a clique of four vertices in $G-v$ and, for each such set $X$, we have $3|X|-6-\sum_{e\in E(G-v)}\deg_{ \{X\}}(e)=6-6=0$.
Hence, we obtain
\begin{align}
\cval(G-v)
&=\sum_{X\in \mathcal{X}_0}(3|X|-6)-\sum_{e\in E(G-v)}(\deg_{\mathcal X_0}(e)-1)\\
&=\sum_{X\in \mathcal{X}_1}(3|X|-6)-\sum_{e\in E(G-v)}(\deg_{\mathcal X_1}(e)-1)\\
&=\sum_{X\in \hXX(G)}(3|X-\{v\}|-6)-\sum_{e\in E(G-v)}(\deg_{\mathcal \hXX(G)}(e)-1).
\end{align}
It follows that $\cval(G)-\cval(G-v)= 3|\hXX_v|-\sum_{e\in E_v}{(\deg_{\hXX(G)}(e)-1)},$ which proves (a).
Part (b) follows by applying (a) to both $G$ and $G[N_G(v)\cup \{v\}]$ and noting that $\hXX_v$ and $E_v$ are the same in both graphs.   			
\end{proof}

Our next three results hold for both the 3-dimensional rigidity and $\cC_2^1$-cofactor matroids of a graph $G$. To unify the proofs we will use the notation that  $\cM\in \{\cR_3,\cC_2^1\}$ and $\rr$ is the rank function of $\cM(G)$. 
Our first result follows immediately from the coning theorems (Theorems \ref{thm:coning} and \ref{thm:coning_cof})
and the fact that $\cR_2=\cC_1^0$.

\begin{lemma}\label{lem:cone0}
Let  $G=(V,E)$ be a  graph, $w\in V$ such that $N_G(w)=V-\{w\}$ 
and $\cM\in \{\cR_3,\cC_2^1\}$. Then 
$G$ is $\cM$-independent (respectively $\cM$-closed, $\cM$-rigid) if and only if $G-w$ is $\cR_2$-independent (respectively $\cR_2$-closed, $\cR_2$-rigid).
\end{lemma}

Part (b) of our next result implies a result of Tay \cite{Taycon} that Conjecture \ref{con:dress} 
holds for cone graphs. Note that, by Lemma \ref{lem:cone0}, a cone graph $G$ with cone vertex $w$ is $\cM$-closed for some $\cM\in \{\cR_3,\cC_2^1\}$ if and only if $G-w$ is $\cR_2$-closed.

\begin{lemma}\label{lem:cone:rcdress}
Let  $G=(V,E)$ be an $\cM$-closed graph for some  $\cM\in \{\cR_3,\cC_2^1\}$. Suppose that $w\in V$ such that $N_G(w)=V-\{w\}$. 
Then the following hold.
\begin{enumerate}[(a)]
\item An edge $e\in E$ is $\mathcal{M}$-redundant if and only if $e\notin F(G)$.
\item $\rr(G)=\cval(G)$. %
\item $\rr(G-w)=\cval(G-w)$.
\end{enumerate}
\end{lemma}
\begin{proof}
(a) Sufficiency follows from the fact that $K_5$ is an $\cM$-circuit. 

To prove necessity first suppose that $e=xy$ is $\cM$-redundant with $w\notin \{x,y\}$. 
Then, by 
Lemma \ref{lem:cone0}, $e$ is $\mathcal{R}_2(G-w)$-redundant, and hence there exists an $\mathcal{R}_2$-circuit $C$ in  $G-w$ such that $e\in E(C)$.
We have $|V(C)|\geq 4$, and, since ${\cal R}_2$-circuits are rigid,
$C$ is $\mathcal{R}_2$-rigid. Therefore, 
Lemma \ref{lem:cone0} implies that $C*w$ is $\mathcal{M}$-rigid. Since $G$ is  $\mathcal{M}$-closed, it follows that $V(C)\cup \{w\}$ induces a clique in $G$ of size at least five that contains $e$. Hence, $e\notin F(G)$.

Next, suppose that  $e=xw$ is an $\cM$-redundant edge. Then there is a vertex $y\in V-\{w\}$ such that $xy$ is also $\cM$-redundant. 
From the previous paragraph, we know that $xy$ is contained in a clique %
of size at least five. Since $w$ is a cone vertex, $e$ is also contained in a clique 
of size at least five and hence $e\notin F(G)$.
\medskip

(b) We proceed by induction on $|F(G)|$. First, suppose $|F(G)|=0.$ Note that $\hXX=\hXX_w$ and $\mathcal{H}=\mathcal{H}_w$. 
By 
Lemma \ref{lem:cone0} and part (a), $G-w$ is $\mathcal{R}_2$-closed and $\cR_2$-bridgeless, and hence every edge of $G-w$ is contained in a clique of $G-w$ of size four. 
Lemma \ref{lem:cone0} and Theorem~\ref{thm:lamanrank}(b) now give
\begin{equation}\label{eq:conerank}
\rr(G)=|V|-1+\rr_2(G-w)=|V|-1+\sum_{X\in \hXX}(2|X|-5).
\end{equation}
Since $G$ is $\cM$-closed, the appropriate gluing lemma
(Lemma~\ref{gluing} or Lemma~\ref{cgluing}) implies that
$|X\cap X'|\leq 2$ for all pairs $X,X'\in \hXX$.
A straightforward counting argument now gives $|V|-1=\sum_{X\in \hXX}(|X|-1) - \sum_{h\in \HH}(\deg_{\hXX}(h)-1).$ 	
Substituting this into \eqref{eq:conerank} gives
$$\rr(G)=\sum_{X\in \hXX}(3|X|-6)-\sum_{h\in \mathcal{H}}(\deg_{\hXX}(h)-1)=\cval(G),$$
as required.

Next, suppose that $|F(G)|\geq 1$. By (a),
each edge $e\in F(G)$ is an $\cM$-bridge.
If there is an edge $e\in F(G)$ that is not incident with $w$, then we can use induction to deduce that $\rr(G)=\rr(G-e)+1=\cval(G-e)+1=\cval(G)$.
Thus, we may assume that each edge $e\in F(G)$ is incident with $w$. Choose $wx\in F(G)$. Then, it follows from (a) 
that $\deg_G(x)=1$. Hence, by induction, we have $\rr(G)=\rr(G-x)+1=\cval(G-x)+1=\cval(G)$.\medskip

(c) The hypothesis that $G$ is $\cM$-closed implies that every induced subgraph of $G$ is $\cM$-closed. In particular, $G-w$ is $\cM$-closed, and (c) will follow from  Theorem \ref{thm:cjt} if $\cM=\cC_2^1$.   Hence we may assume that $\cM=\cR_3$.
In addition, Lemma \ref{lem:cone0}
and the hypothesis that $G$ is $\cR_3$-closed imply that $G-w$ is $\cR_2$-closed.

Consider an $\cR_3$-circuit $C$ of $G-w$. 
By Theorem \ref{dimdrop},
$C$ is $\cR_2$-connected, and hence $\cR_2$-rigid.
Thus the facts that $|V(C)|\geq 5$ and $G-w$ is $\cR_2$-closed imply that
$C$ is contained in a clique $(G-w)[X]$ for some $X\in \hXX(G-w)$.
This in turn implies that the set of non-trivial $\cR_3$-components of $G-w$ is precisely the set $\{(G-w)[X]:X\in \hXX\}$ and ${\cal H}(G-w)=\emptyset$. Since $r_3(G-w)$ is equal to the sum of the ranks of its $\cR_3$-components, this gives $\rr_3(G-w)=|F(G-w)|+\sum_{X\in \hXX(G-w)}(3|X|-6)=\cval(G-w)$, as required.
\end{proof}

We say that a graph  $G=(V,E)$ is {\it locally $\cR_2$-closed} if $G[N_G(v)]$ is $\cR_2$-closed for every $v\in V$. For $\cM\in \{\cR_3,\cC_2^1\}$, it follows from Lemma \ref{lem:cone0} that $G$ is locally $\cR_2$-closed if and only if $G[N_G(v)\cup \{v\}]$ is $\cM$-closed for every $v\in V$. This, and the observation that every induced subgraph of an $\cM$-closed graph is $\cM$-closed, imply that every $\cM$-closed graph is locally $\cR_2$-closed.
\begin{theorem}\label{thm:rcdress:ineq}
Let  $G=(V,E)$ be a locally $\cR_2$-closed graph and $\cM\in \{\cR_3,\cC_2^1\}$.  Then   $\rr(G)\leq \cval(G).$
\end{theorem}			
\begin{proof} We proceed by  induction on $|V|$.  Choose $v\in V$. We have
\begin{align*}
\rr(G) - \rr(G-v) \leq \rr(G[N_G(v)\cup \{v\}]) - \rr(G[N_G(v)]) &= \cval(G[N_G(v)\cup \{v\}]) - \cval(G[N_G(v)])\\ 
&= \cval(G) - \cval(G-v),
\end{align*}
where the inequality follows from the submodularity of $\rr$, the first equality from Lemma \ref{lem:cone:rcdress}(b),(c), and the second equality from Lemma \ref{lem:circuit:NGv}(b). Since $\rr(G-v)\leq \cval(G-v)$ by induction, this gives $\rr(G)\leq \cval(G)$.
\end{proof}

Theorem \ref{thm:rcdress:ineq} shows that, for any $\cR_3$-closed graph $G$, $\cval(G)$  gives an upper bound on $\rr_3(G)$. 
This implies that the right-hand side of the formula in Conjecture
\ref{con:dress} is indeed an upper bound on $\rr_3(G)$.
We close this section by using this result to deduce that Conjectures \ref{con:dress} and \ref{con:walter} are equivalent.

\begin{theorem}
Conjectures \ref{con:dress} and \ref{con:walter}
are equivalent.
\end{theorem}

\begin{proof}
As noted in the Introduction, Theorem \ref{thm:cjt} tells us that Conjecture \ref{con:dress} would follow from Conjecture \ref{con:walter}. To see the other direction, let us assume that Conjecture \ref{con:dress} is true. Choose an arbitrary graph $G$ and let $\bar G$ be its $\cR_3$-closure. Then 
$\rr_3(G)=\rr_3(\bar G)=\cval(\bar G)$ since Conjecture \ref{con:dress} holds. On the other hand, $\bar G$ is locally $\cR_2$-closed so $\rr_2^1(\bar G)\leq\cval(\bar G)$ by Theorem \ref{thm:rcdress:ineq}.
This gives 
$\rr_2^1(G)\leq \rr_2^1(\bar G)\leq \cval(\bar G)=\rr_3(G).$
We can now use the fact that $\cC_2^1$ is the unique maximal abstract 3-rigidity matroid by \cite{CJT1} to deduce that $\rr_2^1(G)=\rr_3(G).$
\end{proof}

\section{Rank contributions of vertices in matroids on the edge set of a complete graph}\label{sec:rc}

Let $\cM=(E(K_n),r)$ be a matroid defined on the edge set of the complete graph $K_n$ and  $G=(V,E)$ be a subgraph of $K_n$. We will denote the restriction of $\cM$ to $E$ by $\cM(G)$ and the rank of $\cM(G)$ by $\rr(G)$. Given an  ordering $\pi$ of $V$ and a  vertex $v\in V$, let $T_v^\pi$ denote the set of those vertices which precede $v$ in $\pi$. 
We define the {\it rank contribution of $v$ in $\cM(G)$ with respect to $\pi$}  to be $$\rc(G, v, \pi) = \rr(G[T_v^\pi\cup \{v\}])-\rr(G[T_v^\pi]).$$

We now suppose that $\pi$ has been chosen from the uniformly random distribution of orderings of $V$ and define the {\it rank contribution of $v$ in $\cM(G)$} to be

$$\rc(G,v)= \E\big(\rc(G,v,\pi)\big).$$

We first show that the rank of $\cM(G)$ is equal to the sum of the rank contributions of the vertices of $G$. As a warm up observation, note that this equality holds if $G$ is $\cM$-independent, since in this case $\rc(G,v)=\frac{\deg_G(v)}{2}$ for all $v\in V$.

\begin{lemma} \label{lem:rc_sum}
For every subgraph $G$ of $K_n$, we have ${\displaystyle \rr(G)=\sum_{v\in V} \rc(G,v)}.$ 
\end{lemma} 
\begin{proof} The definition of $\rc(G, v, \pi)$ implies that $\rr(G)=\sum_{v\in V}\rc(G, v, \pi)$ for each ordering $\pi$ of $V$. This gives
$$ \rr(G)=\E(\rr(G))=\E\left(\sum_{v\in V}\rc(G, v, \pi)\right)=\sum_{v\in V}\E\left(\rc(G, v, \pi)\right)= \sum_{v\in V} \rc(G,v).$$
\end{proof}

Our next lemma 
quantifies the effect of adding an edge to $G$ on the rank contribution of a vertex. 
\begin{lemma}\label{lem:new_edge}
Let $G=(V,E)$ be a graph and $v, x,y$ be distinct vertices of $G$. Then 
\begin{enumerate}[(a)]
\item $\rc(G,v)\geq \rc(G+xy,v)$ and 
\item $\rc(G,v)\leq \rc(G+xv,v)$. 
\end{enumerate}
\end{lemma}
\begin{proof}
We first prove (a). By the submodularity of $\rr$, we have
$$ \rr(G[T_v^\pi\cup \{v\}])+\rr((G+xy)[T_v^\pi])\geq \rr(G[T_v^\pi]) + \rr((G+xy)[T_v^\pi\cup  \{v\}])$$ for each ordering $\pi$ of $V$. This implies that $\rc(G,v,\pi)\geq \rc(G+xy,v,\pi)$ and hence  $\rc(G,v)\geq \rc(G+xy,v).$ 

Part (b) follows easily from the monotonicity of $\rr$.
\end{proof}

For $G\subseteq K_n$, let $G^+$ be the spanning subgraph of $K_n$ obtained by adding the vertices in $V(K_n)\setminus V(G)$ to $G$ as isolated vertices. It is straightforward to check that $\rc(G^+,v)=\rc(G,v)$ %
for all  vertices $v$ of $G$. Combined with Lemma \ref{lem:new_edge}(a) this gives

\begin{corollary}\label{cor:new_edge}
Suppose $G\subseteq H\subseteq K_n$, $v$ is a vertex of $G$ and $N_G(v)=N_H(v)$. Then $\rc(G,v)\geq \rc(H,v)$. 
\end{corollary}

Given a vertex $v$ of a graph $G=(V,E)$ and an ordering   $\pi$ of $V$, put $N_v^\pi=T_v^\pi\cap N_G(v)$.
We next describe the distribution of $N_v^\pi$  over all  orderings $\pi$, chosen uniformly at random. Let $k=\deg_G(v)$. 
First note that for any $i\in\{0,1,\dots, k\}$, 
\begin{equation}\label{eq:b0}
\Prob(|N_v^\pi|=i)=\frac{1}{k+1}
\end{equation}
since $\pi$ induces a uniformly random ordering on the set $N_G(v)\cup\{v\}.$, and $\frac{1}{k+1}$ is the probability of the event that $v$ takes the $(i+1)$'th position in this ordering. 
Given that $|N_v^\pi|=i$, the set $N_v^\pi$ is uniformly distributed across the subsets of $N_G(v)$ of size $i$.  Thus, for any $U\subseteq N_G(v)$, we have $$\Prob\bigg(N_v^\pi=U \;\bigg|\; |N_v^\pi|=|U|\bigg)=\frac{1}{\binom{k}{|U|}} $$ and hence $$ \Prob\big(N_v^\pi=U\big)=\frac{1}{k+1}\cdot\frac{1}{\binom{k}{|U|}}\;.$$

\subsection*{Matroids with the 0- and 1-extension properties}

Although the rank contribution of vertices was not explicitly defined in \cite{vill}, the following lemmas (in slightly weaker forms) are {implicitly contained in the proof of 
\cite[Lemma 3.2]{vill}}. We reproduce their proofs for the sake of completeness.

We say that a matroid $\cM$ on $E(K_n)$ has the {\em $d$-dimensional 0-extension property} if, for every graph $G\subseteq K_n$ and every vertex $v$ of $G$ of degree at most $d$, $G$ is $\cM$-independent whenever $G-v$ is $\cM$-independent.

A vertex $v$ of a graph $G$ is said to be {\em simplicial} if $N_G(v)$ induces a clique in $G$.

\begin{proposition}\label{prop:rc:probsum}
Let $\cM$ be a matroid on $E(K_n)$ which has the $d$-dimensional $0$-extension property, $ G=(V,E)$ be a subgraph of $K_n$, and $v\in V$ with $\deg_G(v)=k$.
Then
$$\rc(G, v)= \begin{cases}
\frac{k}{2} &  \text{ if } k\leq d, \\
d-\frac{1}{k+1}{\binom{d+1}{2}} +\sum_{i=d+1}^k \Prob(\rc(G,v,\pi)\geq i) &\mbox{ if $k\geq d$.}
\end{cases}
$$
In addition, if $k\geq d$, every copy of $K_{d+2}$ is a circuit in $\cM$ and $v$ is a simplicial vertex of $G$, then  $\rc(G, v)= d-\frac{1}{k+1}{\binom{d+1}{2}}$.
\end{proposition}
\begin{proof} If $k\leq d$ then the assertion that $\rc(G, v)=\frac{k}{2}$ follows from the $0$-extension property and the fact that $v$ is equally likely to occur in every position of a uniformly random ordering of $N_G(v)\cup \{v\}$.
So let us assume that $k\geq d$.

We use  
the fact that, for a non-negative integer valued random variable $X$, we have $\E X=\sum_{i=1}^\infty \Prob(X\geq i)$. This implies that 
\begin{equation}\label{eq:egy}
\rc(G,v)=\sum_{i=1}^k \Prob(\rc(G,v,\pi)\geq i)=\sum_{i=1}^d \Prob(\rc(G,v,\pi)\geq i)+\sum_{i=d+1}^k \Prob(\rc(G,v,\pi)\geq i).
\end{equation}
We can now use the hypothesis that $\cM$ has the 0-extension property and (\ref{eq:b0}) to deduce that,
$$\sum_{i=1}^d \Prob(\rc(G,v,\pi)\geq i)=\sum_{i=1}^d \Prob(|N_v^\pi|\geq i)%
=\sum_{i=1}^d \frac{k-i+1}{k+1}=d-\frac{1}{k+1}{\binom{d+1}{2}}.$$
Substituting this equation into (\ref{eq:egy}) completes the proof of the first part of the lemma. The second part follows since the hypotheses that every copy of $K_{d+2}$ is a circuit in $\cM$ and $N_G(v)$ induces a clique in $G$ imply that $\rc(G,v,\pi)\leq d$ for all orderings $\pi$ of $V$.
\end{proof}

We say that a matroid $\cM$ on $E(K_n)$ has the {\em $d$-dimensional 1-extension property} if, for every graph $G\subseteq K_n$ and every vertex $v$ of $G$ of degree $d+1$, $G$ is $\cM$-independent whenever $G-v+uw$ is $\cM$-independent for some distinct $u,w\in N_G(v)$ with $uw\notin E(G)$. The next statement gives an improved bound on the rank contribution of a vertex when $\cM$ has the $d$-dimensional 1-extension property.

\begin{lemma}
\label{lem:uglyboundnew}
Let $\cM$ be a matroid on $E(K_n)$ which has the $d$-dimensional $0$- and $1$-extension properties, $ G=(V,E)$ be a subgraph of $K_n$ and $v\in V$ with $\deg_G(v)=k\geq {d}$.
Suppose that 
\begin{enumerate}[(i)] 
\item the graph $G-v$ is $\cM$-closed; 
\item $N_G(v)$ does not induce a clique in $G$, and $|V(H_1) \cap V(H_2)| \leq d - 2$  for any two distinct 
maximal cliques $H_1$ and $H_2$ of $G[N_G(v)]$.
\end{enumerate}
Then $\rc(G,v)\geq d-\frac{1}{k}{\binom{d+1}{2}}+\frac{1}{2}$.
\end{lemma}
\begin{proof}
The hypotheses that $\mathcal{M}$ has the 1-extension property and $G-v$ is $\cM$-closed imply that, for each ordering $\pi$ of $V$, $\rc(G,v,\pi)\geq d+1$ whenever $|N_v^\pi|\geq d+1$ and $G[N_v^\pi]$ is not a clique. In addition, condition (ii) 
implies that this event occurs with probability at least $ \frac{1}{2}\left(1-\frac{d(d+1)}{k(k+1)}\right),$ see the proof of \cite[Lemma 3.2]{vill} for more details.
Lemma~\ref{prop:rc:probsum} now gives
\begin{equation*}
\rc(G,v)
\geq  d-\frac{1}{k+1}{\binom{d+1}{2}}+\frac{1}{2}\left(1-\frac{d(d+1)}{k(k+1)}\right)
=d+\frac{1}{2}-\frac{1}{k}{\binom{d+1}{2}}.
\qedhere	
\end{equation*}
\end{proof}

\begin{corollary}\label{coro:rc:geq:d}
Let $ G=(V, E)$ be a graph, $v\in V$ with $\deg_G(v)=k\geq d$, and
		$\cM\in\{\cR_d,\cC_{d-1}^{d-2}\}$.
        Suppose that $G+K(N_G(v))$ is $\mathcal{M}$-rigid but $G$ is not.
		Then $\rc(G,v)\geq d+\frac{1}{2}-\frac{1}{k}{\binom{d+1}{2}}.$
\end{corollary}
	\begin{proof}
	Let $G'$ be the supergraph of $G$ obtained by adding every edge $xy$ that belongs to the $\mathcal{M}$-closure of $G$ and is not incident with $v$.
    Then $G'$ satisfies the assumptions of Lemma \ref{lem:uglyboundnew}, and $\deg_{G'}(v)=\deg_G(v)$. Hence,
    $\rc(G',v)\geq d+\frac{1}{2}-\frac{1}{k}{\binom{d+1}{2}}$.
    The result now follows by applying 
    Corollary \ref{cor:new_edge}.
	\end{proof}

\section{Rank contributions in $\mathcal{R}_2$}

We will apply the results of the last section, taking $\cM$ to be the generic 2-dimensional rigidity matroid $\cR_2$, 
to derive a closed formula for the rank contribution  $\rc_2(G,v)$ of a vertex $v$ in a graph $G$
to $\cR_2(G)$. This will allow us to introduce new proof techniques, which we will subsequently apply to $\cR_3$ and $\cC_2^1$, in the more straightforward context of $\cR_2$.

Our first result gives an expression for the rank contribution of a vertex in an $\cR_2$-closed graph. It can be viewed as a local version of Theorem \ref{thm:lamanrank}(b). Indeed,  Theorem \ref{thm:lamanrank}(b)  follows from a  combination of this result and Lemma \ref{lem:rc_sum}.

\begin{lemma}
\label{2dimlocal}
Let $G=(V,E)$ be an $\mathcal{R}_2$-closed graph, $v$ be a non-isolated vertex of $G$, and $\mathcal{X}_v(G)$ be the vertex sets of the maximal cliques of $G$ that contain $v$.
Then $$\rc_2(G,v)=\sum_{X\in \mathcal{X}_v(G)}\Big(2-\frac{3}{|X|}\Big).$$
\end{lemma}

\begin{proof} %
We may assume, without loss of generality, that $G$ has no isolated vertices. Let $\cX(G)$ denote the vertex sets of the maximal cliques of $G$.

We first consider the case when  $v$ is a cone vertex of $G$, i.e., $uv\in E$ for all $u\in V-\{v\}$. 
Then $\mathcal{X}(G)=\mathcal{X}_v(G)$. 
Since $G$ is $\mathcal{R}_2$-closed, Lemma \ref{gluing} implies that $X\cap Y=\{v\}$ for all distinct $X,Y\in \XX(G)$, and there is no edge joining $X-\{v\}$ and $Y-\{v\}$. Thus, for each vertex $u\in V-\{v\}$, $N_G(u)$ induces a clique in $G$ and we have $\cX_u=\{X_u\}$ where $X_u=N_G(u)\cup \{u\}$. Lemma \ref{prop:rc:probsum} 
now implies that $\rc_2(G,u)=2-\frac{3}{|X_u|}$ for all $u\in V-\{v\}$. In addition,
we can use Lemma \ref{lem:rc_sum} to deduce that
$$\rc_2(G,v)=\rr_2(G)-\sum_{u\in V-\{v\}}\rc_2(G,u)=\rr_2(G)-\sum_{\substack{X\in \XX(G)\\ u\in X-\{v\}}}\Big(2-\frac{3}{|X|}\Big)=\sum_{X\in \mathcal{X}_v(G)}\Big(2-\frac{3}{|X|}\Big),$$
where the last equality follows from Theorem \ref{thm:lamanrank}(b) and the fact that $\mathcal{X}(G)=\mathcal{X}_v(G)$.

We next turn to the general case. Choose an arbitrary vertex $v\in V$. Let $G_v=G[N_G(v)\cup \{v\}]$.
Note that the maximal cliques containing $v$ in $G$ and in $G_v$ are
identical. Furthermore, we have
$\rc_2(G,v) \leq \rc_2(G_v,v)$ by Corollary \ref{cor:new_edge}. 
We can now use the argument of the previous paragraph to obtain
\begin{equation}\label{eq:rcNGV:aux}
\rc_2(G,v) \leq \rc_2(G_v,v) = \sum_{X\in \XX_v(G)}\Big(2-\frac{3}{|X|}\Big)
\qquad \text{ for every } v\in V.
\end{equation}
Lemma \ref{lem:rc_sum} now gives 
\begin{equation}
\label{eq:rcNGV:aux2}
\rr_2(G)=\sum_{v\in V}\rc_2(G,v)\leq \sum_{v\in V}\sum_{X\in \mathcal{X}_v(G)}\Big(2-\frac{3}{|X|}\Big)=\rr_2(G),
\end{equation}
where the last equality follows from Theorem \ref{thm:lamanrank}(b). Therefore,  the inequalities in \eqref{eq:rcNGV:aux}
and \eqref{eq:rcNGV:aux2}
must hold with equality for all $v\in V$. This completes the proof.
\end{proof}

Lov\'asz and Yemini \cite{LY} proved that every 6-connected graph is rigid in $\R^2$.
We close this section by showing that this  sufficient connectivity condition for $\cR_2$-rigidity can be improved for $K_3$-covered graphs. We also deduce a
stronger, globally rigid version.

\begin{theorem}
\label{K3}
Let $G=(V,E)$ be a 3-connected, $K_3$-covered graph. Then\\
(a) $G$ is rigid in $\R^2$,\\
(b) $G$ is globally rigid in $\R^2$.
\end{theorem}

\begin{proof}
(a) Suppose, for a contradiction, that the theorem is false and that $G$ is a counterexample with as few vertices as possible and, subject to this condition, as many edges as possible. Then
$|V|\geq 5$ and
$G$ is $\cR_2$-closed. For each $v\in V$, let ${\cal X}_v(G)$ be the set of maximal cliques of $G$ that contain  $v$.
If  $|{\cal X}_v(G)|=1$ for some $v\in V$,
then $G-v$ is a 3-connected, $K_3$-covered graph and we may apply  induction to deduce that $G-v$ is rigid. This would imply that  $G$ is also rigid and give the required contradiction.

Hence we may assume that $|{\cal X}_v(G)|\geq 2$ for all $v\in V$. 
The fact that $G$ is $K_3$-covered implies that $|X|\geq 3$ for all $X\in {\cal X}_v(G)$.
Hence we have
$$\rc_2(G,v)=\sum_{X\in \mathcal{X}_v(G)}\Big(2-\frac{3}{|X|}\Big)\geq |{\cal X}_v(G)|\geq 2$$
for all $v\in V$. This gives $\rr_2(G)\geq 2|V|$ by Lemma \ref{lem:rc_sum},
and contradicts (\ref{gl}). This completes the proof of (a).

(b)
Let $G^+$ be the graph obtained from $G$ by adding a new vertex $v$ and three new edges $va,vb,vc$, for each triple $\{a,b,c\}$ of $V$ that induces a triangle in $G$. Then $G^+$ is $3$-connected and redundantly rigid in $\R^2$
by (a). Thus $G^+$ is globally rigid in $\R^2$ by \cite{JJconnrig}. 
Since $G[N_{G^+}(v)]$ is
a complete subgraph of $G$ for
each $v\in V(G^+)-V(G)$, it follows that $G$ is also globally rigid in $\R^2$.
\end{proof}

Theorem \ref{K3}(b) implies that a 2-dimensional combinatorial zeolite is globally rigid if and only if its
underlying graph is 3-edge-connected, see \cite[Corollary 3.4]{Jzeo}.

\section{Rank contributions in $\mathcal{R}_3$ and $\mathcal{C}_2^1$ }

We will use Theorem \ref{thm:cjt} to derive a closed formula for the rank contribution $\rc_2^1(G,v)$ of a vertex $v$ in a $\mathcal{C}_2^1$-closed graph $G$ to $\cC_2^1(G)$, and show that a similar formula gives an upper bound on the rank contribution $\rc_3(G,v)$ of a vertex $v$ in a $\mathcal{R}_3$-closed graph $G$ to $\cR_3(G)$.

For an arbitrary  graph $G=(V,E)$ and vertex $v\in V$, recall the definitions of $F(G)$, $\hat{\mathcal{X}}(G)$, $\mathcal{H}(G)$, $\cval(G)$, $F_v(G)$,  $\hXX_v(G)$ and $\mathcal{H}_v(G)$ from Section \ref{sec:dress}.
We  define the {\em clique value of $G$ at $v$} to be 
$$\cval(G,v):=\frac{|F_v|}{2}+\sum_{X\in \hXX_v}\left(3-\frac{6}{|X|}\right)-\sum_{h\in \mathcal{H}_v}\left(\frac{\deg_{\hXX}(h)-1}{2}\right),$$
when $v$ is not an isolated vertex of $G$, and to be zero otherwise.
Then, a straightforward  calculation gives 
\begin{equation}\label{eq:kette}
\sum_{v\in V}\cval(G,v)=\cval(G).
\end{equation}

Theorem \ref{thm:rc:cofactor} below gives a local version of Theorem \ref{thm:cjt} by showing that $\rc_2^1(G,v)=\cval(G,v)$ for any vertex $v$ in a $\cC_2^1$-closed graph $G$. We will also show in Theorem \ref{thm:rcdress:ineq:rc} that  $\cval(G,v)$ gives an upper bound on  $\rc_3(G,v)$  for any vertex $v$ in an $\cR_3$-closed graph $G$.

Our next three results  hold for  a graph $G$ and each matroid  $\cM\in \{\cR_3(G),\cC_2^1(G)\}$. We will continue to use $\rr$ and $\rc$ for the rank and rank contribution functions of $\cM$. Our first lemma determines $\rr(G)$ and $\rc(G,v)$ for a special family of graphs,
illustrated in Figure \ref{fig:spec}.

\begin{figure}[!t]
        \begin{center}
            \vspace{-0.4cm}
            \includegraphics[scale=0.4]{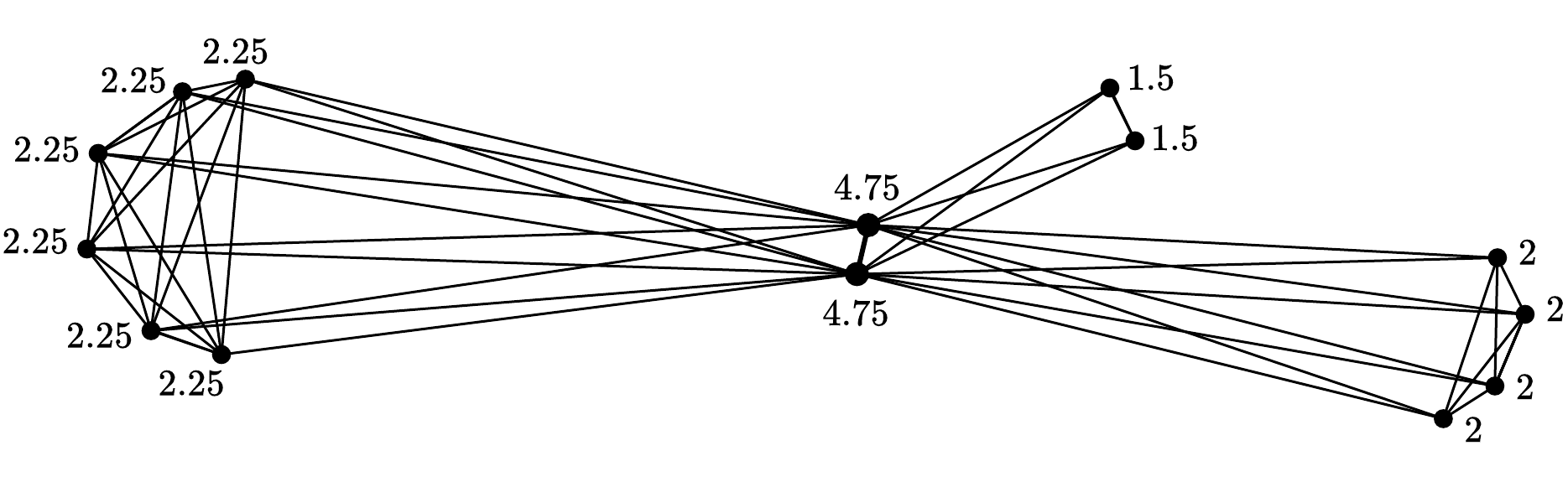}
            \vspace{-0.3cm}
\caption{Illustration of Lemma \ref{lem:cliques}. The graph has three maximal cliques of sizes \(4\), \(6\), and \(8\), sharing two common vertices. Its total rank in both \(\mathcal{C}_2^1\) and \(\mathcal{R}_3\) is \(34\), and each vertex is labeled with its rank contribution.}            \label{fig:spec}
        \end{center}
    \end{figure}

\begin{lemma}\label{lem:cliques}
Let  $G$ be a graph with  maximal cliques $C_1, C_2,$ $\dots,$ $C_k$, 
and  $\cM\in\{\cR_3(G),\cC_2^1(G)\}$. Suppose that $V(C_i)\cap V(C_j)=\{w_1,w_2\}$ for all $1\le i< j \leq k$.
Then $\rr(G)=\cval(G)$ and $\rc(G,v)=\cval(G,v)$ for all $v\in V$.
\end{lemma}
\begin{proof} The result is immediate when $|V|\leq2$. Hence we may
assume that every $C_i$ has at least three vertices. It is straightforward to check that $\cval(G)=3|V|-5-k=\rr(G)$. Then (\ref{eq:kette}) and Lemma \ref{lem:rc_sum} give 
\begin{equation}\label{eq:hingeexample}
\sum_{v\in V}\cval(G,v)=\cval(G)=3|V|-5-k=\rr(G)=\sum_{v\in V}\rc(G,v).
\end{equation}
If $v\in V(C_i)-\{w_1,w_2\}$, then $v$ is simplicial, and we can use Lemma \ref{prop:rc:probsum}  to deduce that  %
$\cval(G,v)=3-\frac{6}{|V(C_i)|}=\rc(G,v)$.
Together with \eqref{eq:hingeexample}, this implies that
$\sum_{i=1}^2\rc(G,w_i)=\sum_{i=1}^2\cval(G,w_i).$
By symmetry, $\rc(G,w_i)=\cval(G,w_i)$ follows for both $i=1,2$.
\end{proof}

Our next result concerns rank contributions in closed cone graphs.

\begin{lemma}\label{lem:cone:rcdress:rc}
Let  $G=(V,E)$ be an $\cM$-closed graph for some  $\cM\in \{\cR_3,\cC_2^1\}$. Suppose $N_G(w)=V-\{w\}$ for some $w\in V$. 
Then, for each $v\in V$:
\begin{enumerate}[(a)]
\item 
$\rc(G,v)=\rc(G[N_G(v)\cup \{v\}],v)$;
\item $\displaystyle\rc(G,v)=\cval(G,v)$.
\end{enumerate}
\end{lemma}
\begin{proof}
(a) Put $G_v=G[N_G(v)\cup \{v\}].$ 	
For an arbitrary ordering $\pi$ of $V$, let $\pi_v$ be the restriction of $\pi$ to $N_G(v)\cup \{v\}$.
We will show  that $\rc(G, v,\pi)=\rc(G_v, v,\pi_v)$,
i.e.,
$$ \rr(G[T_v^\pi\cup \{v\}])-\rr(G[T_v^\pi])=\rr(G[(N_G(v)\cap T_v^\pi)\cup \{v\}])-\rr(G[N_G(v)\cap T_v^\pi]).$$ 
If $w\in T_v^\pi$, then this follows by applying Lemma \ref{lem:cone:rcdress}(b) and Lemma \ref{lem:circuit:NGv}(b) to the graph $G[T_v^\pi\cup \{v\}]$.
If $w\notin T_v^\pi$, then it follows similarly from Lemma \ref{lem:cone:rcdress}(c) and Lemma \ref{lem:circuit:NGv}(b). 
The fact that, for every ordering $\pi$, we have $\rc(G, v,\pi)=\rc(G_v, v,\pi_v)$ implies that $\rc(G,v)=\rc(G_v,v)$, as required.

\medskip

(b) 
We first suppose that $v\neq w$. Let $G_v=G[N_G(v)\cup \{v\}]$. Observe that $\cval(G,v)=\cval(G_v,v)$. 
In addition, $G_v$ satisfies the hypotheses of Lemma \ref{lem:cliques}. (This follows since $G_v$ is an $\cM$-closed, double cone graph with cone vertices $v,w$ and hence $G_v-v-w$ is $\cR_1$-closed, so is a disjoint union of cliques.)
Lemma \ref{lem:cliques} now gives $\cval(G_v,v)=\rc(G_v,v)$. 
Hence, 
$$\cval(G,v)=\cval(G_v,v)=\rc(G_v,v)=\rc(G,v),$$ where the last equality follows from (a).

It remains to consider the case when $v=w$. Using Lemma \ref{lem:cone:rcdress}(b) and Lemma \ref{lem:rc_sum}, we obtain $$\sum_{v\in V}\cval(G,v)=\cval(G)=\rr(G)=\sum_{v\in V}{\rc(G,v)}.$$
Since $\cval(G,v)=\rc(G,v)$ for all $v\neq w$, it follows that $\cval(G,w)=\rc(G,w)$.
\end{proof}

\begin{theorem}\label{thm:rcdress:ineq:rc}
Let  $G=(V,E)$ be a locally $\cR_2$-closed graph, $v\in V$ and  $\cM\in \{\cR_3,\cC_2^1\}$.  
Then  $\rc(G,v)\leq \cval(G,v).$
\end{theorem}
\begin{proof} Let $G_v=G[N_G(v)\cup \{v\}]$. 
Since $G$ is locally $\cR_2$-closed, $G_v$ is $\cM$-closed. Thus, $\cval(G_v,v)=\rc(G_v,v)$ by Lemma \ref{lem:cone:rcdress:rc}(b). Corollary \ref{cor:new_edge} and (\ref{eq:kette}) now give
\begin{equation*}
    \cval(G,v)=\cval(G_v,v)=\rc(G_v,v)\geq \rc(G,v).\qedhere
\end{equation*}
\end{proof}

We close this section by using Theorem \ref{thm:cjt} to show that the upper bound on the rank contribution of a vertex in a $\cC_2^1$-closed graph given by Theorem \ref{thm:rcdress:ineq:rc} is tight.

\begin{theorem}\label{thm:rc:cofactor}
Let $G=(V,E)$ be a $\cC_2^1$-closed graph, $v\in V$. Then $\rc_2^1(G,v)= \cval(G,v).$
\end{theorem}
\begin{proof}
Every $\cC_2^1$-closed graph is locally $\cR_2$-closed,
as observed in Section~\ref{sec:dress}. 
By Theorem \ref{thm:rcdress:ineq:rc}, for each $v\in V$, we have $\rc_2^1(G,v)\leq \cval(G,v)$.
Theorem \ref{thm:cjt} now gives 
$$\rr_2^1(G)=\sum_{v\in V}\rc_2^1(G,v)\leq  \sum_{v\in V}\cval(G,v)=\cval(G)=\rr_2^1(G).$$
Hence each inequality $\rc_2^1(G,v)\leq \cval(G,v)$ must hold with equality.
\end{proof}

\section{$\mathcal{C}_2^1$-rigidity of $K_4$-covered graphs} 
\label{sec:cof}

We use rank contributions to analyse the $C_2^1$-cofactor matroid of a $K_4$-covered graph. We first show that every 6-connected,  $K_4$-covered graph $G$ is {\em $\cC_2^1$-rigid}, i.e. $r_2^1(G)=3|V|-6$.  We 
then show that the $C_2^1$-cofactor version of a conjectured characterisation of the rank of a body-pin graph in the 3-dimensional rigidity matroid is valid.
We will need the following result which uses Theorem \ref{thm:rc:cofactor} to derive new
formulas and bounds for the rank 
contribution of a vertex $v$ in a $\cC_2^1$-closed graph $G$ in terms of the set $\hat\cX_v$ defined in Section \ref{sec:DW}.

\begin{lemma}\label{lem:cof:bounds}
Let $G=(V,E)$ be a $\cC_2^1$-closed graph, $v\in V$, and put $H=G[N_G(v)]$.  Let $J$ be the set of edges of $H$ which are not contained in a copy of $K_4$ in $H$.\\
(a)        $$\displaystyle \rc_2^1(G,v)=\frac{\deg_G(v)}{2}+\sum_{X\in\hXX_v}\bigg(\frac{7}{2}-\frac{6}{|X|}-\frac{|X|}{2}\bigg).$$
(b) If $\deg_G(v)\geq3$, then
$$\rc_2^1(G,v)= \frac{2\deg_G(v)-\rr_2(H)+|J|}{4}+\sum_{X\in \hXX_v}\bigg(\frac{9}{4}-\frac{6}{|X|}\bigg)=\frac{\dof_2(H-J)+3}{4}+\sum_{X\in \hXX_v}\bigg(\frac{9}{4}-\frac{6}{|X|}\bigg).$$
(c) If $v$ is not a simplicial vertex of $G$, then $\rc_2^1(G,v)\geq 1+ \frac{21}{20} \cdot |\hXX_v|.$
\end{lemma}
\begin{proof}
(a) 
A simple counting argument based on the definition of $F_v$ in Section \ref{sec:DW} gives 
$$|F_v|=\deg_G(v)-\left|\bigcup_{X\in \hXX_v} (X-v)\right|=\deg_G(v)-\sum_{X\in \hXX_v}(|X|-1) + \sum_{h\in \HH_v}(\deg_{\hXX}(h)-1).$$
Theorem
\ref{thm:rc:cofactor} and the definition of $\cval(G,v)$ now give (a).\medskip

(b)  Since $G$ is $\cC_2^1$-closed,  $H$ is $\cR_2$-closed by Lemma \ref{lem:cone0}. We can now apply
Theorem \ref{thm:lamanrank}(b) to $H$ and obtain 
$\rr_2(H)=\sum_{Y\in\XX(H)}(2|Y|-3)=|J|+\sum_{X\in\hXX_v}(2|X-\{v\}|-3)$, where $\XX(H)$ is the set of maximal cliques of $H$ of size at least two.
Hence $$\frac{\rr_2(H)-|J|}{4}=\sum_{X\in\hXX_v}\Big(\frac{|X|}{2}-\frac{5}{4}\Big)$$ holds.
Combining this equality with (a) gives 
the first equality in 
(b).
Since $H$ is $\cR_2$-closed,
every edge in $J$ is an $\cR_2$-bridge of $H$ and hence 
$$\dof_2(H-J)=2|V(H)|-3-\rr_2(H-J)=2|V(H)|-3-\rr_2(H)+|J|.$$
Thus the second equality of (b) follows by combining the first equality with the above expression for  $\dof_2(H-J)$.

\medskip

(c) Suppose $v$ is not a simplicial vertex of $G$. Then the fact that $G$ is $\cC_2^1$-closed,
together with Theorem \ref{thm:coning_cof}, 
implies that $H$ is not $\mathcal{R}_2$-rigid. Hence $2\deg_G(v)-\rr_2(H)\geq 4$. 
Thus the first term of the RHS of the first equality of (b) is at least $1$. Since $|X|\geq 5$ for
all $X\in \hXX_v$, the contribution of each set $X$ to the second term is at least
$\frac{9}{4}-\frac{6}{5}=\frac{21}{20}$. Hence
(c) follows from the first equality in (b).
\end{proof}

\subsection{\boldmath $C_2^1$-rigidity of 6-connected $K_4$-covered graphs}
We derive a sufficient condition for the $\cC_2^1$-rigidity of a $K_4$-covered graph and give an example which shows that our condition is best possible.

\begin{theorem}\label{thm:zeo:general}
Every 6-connected, $K_4$-covered graph is $\cC_2^1$-rigid.
\end{theorem}

\begin{proof}
Suppose, for a contradiction,  that $G=(V,E)$ is a counterexample such that $|V|$ is as small as possible and,
subject to this condition, $|E|$ is as large as possible. It follows from the maximality of $|E|$ and Theorem \ref{thm:CJT:K5} that $G$ is $\cC_2^1$-closed.
Furthermore,   the minimality of $|V|$ implies that  
$G$ has no simplicial vertices.
Let us call a vertex $v\in V$ {\it deficient} if $\rc_2^1(G,v)<3$ holds.

\begin{claim}\label{claim:deficient}
Let $v\in V$ be a deficient vertex. Then:
\begin{enumerate}[(a)]
\item $\deg_G(v)=6$ and $\rc_2^1(G,v)=3-\frac{1}{5}$;
\item $G[N_G(v)]$ is the union of a copy of $K_4$ and a copy of $K_3$ with exactly one common vertex, which we denote by $v^*$; 
\item $\deg_G(v^*)\geq 8.$
\end{enumerate}
\end{claim}
\begin{proof} It follows from the 6-connectivity of $G$ and Lemma \ref{lem:cof:bounds}(a) that $|\hXX_v|\geq 1$. On the other hand, Lemma \ref{lem:cof:bounds}(c) implies that $|\hXX_v|\leq 1.$ Hence
$|\hXX_v|= 1$. Let $\hXX_v=\{X\}$. Since $G$ is a $\cC_2^1$-closed, $K_4$-covered graph and $v$ is not a simplicial vertex of $G$, the Gluing Lemma for cofactor matroids 
(Lemma \ref{cgluing}) implies that $\deg_G(v)\geq |X|+1$. Thus, by Lemma \ref{lem:cof:bounds}(a), 
$$\rc_2^1(G,v)=\frac{\deg_G(v)}{2}+\Big(\frac{7}{2}-\frac{6}{|X|}-\frac{|X|}{2}\Big)\geq 4-\frac{6}{|X|}.$$
This bound and
$\rc_2^1(G,v)<3$ imply that $|X|=5$ and $\deg_G(v)=6$, from which (a) follows.
The fact that $G$ is a $\cC_2^1$-closed $K_4$-covered graph now implies (b).

To prove (c), observe that the 6-connectivity of $G$ implies that $\deg_G(v^*)\geq 7$, since otherwise $N_G(v)\setminus \{v^*\}$ would be a vertex separator of size five in $G$. 
Let $H_1$ and $H_2$ denote the copies of $K_5$ and $K_4$ in $G[N_G(v)\cup \{v\}]$. 

Suppose that $\deg_G(v^*)=7$. Let $w$ be the unique neighbour of $v^*$
which does not belong to $H_1$ or $H_2$. Since $G$ is a $K_4$-covered graph, the edge
$wv^*$ belongs to a $K_4$-subgraph $H$. 
The assumption that $\deg_G(v^*)=7$ now gives $H-w\subseteq H_1\cup H_2$.
By (b), $H$
cannot intersect both $H_1-H_2$ and $H_2-H_1$. Hence $H$ intersects $H_i$
in three vertices for some $i=1,2$.  The Gluing Lemma for cofactor matroids now implies that $H\cup H_i$ is $C_2^1$-rigid, contradicting the facts that $H_i$ is a maximal clique in $G$ and $G$ is $\cC_2^1$-closed. This proves (c). 
\end{proof}

Let $D$
be the set of deficient vertices in $G$, and let $S=V\setminus D.$ %
By Claim \ref{claim:deficient}, the map $v\mapsto v^*$ defines a function $\sigma:D\to S$. For each $u\in S$, let $\sigma^{-1}(u)=\{v\in D : v^*=u\}$.
We shall show that the rank contribution of each vertex $u\in S$ is 
sufficiently large to ensure that the average rank contribution of the vertices of $G$ is at least three.

\begin{claim}\label{claim:surplus}
For each vertex  $u\in S$, \begin{equation}\label{eq:3+s/5}
\rc_2^1(G,u)\geq 3+\frac{|\sigma^{-1}(u)|}{5}.
\end{equation}
\end{claim}

\begin{proof}
If $\sigma^{-1}(u)=\emptyset$, then the statement follows from the fact that $u$ is not deficient. Hence we may assume  that $|\sigma^{-1}(u)|\geq 1$. 
Let $H=G[N_G(u)]$, and let $J$ denote the set of edges induced by $N_G(u)$ that are not contained in a copy of $K_4$ in $H$.
It follows from Claim \ref{claim:deficient}(b) that each vertex in $ \sigma^{-1}(u)$ is incident with two edges in $J$. Since each edge in $J$ is incident with at most two vertices in $\sigma^{-1}(u)$, this gives
$|J|\geq |\sigma^{-1}(u)|$. 
Furthermore, since $G$ is $\cC_2^1$-closed and $N_G(u)$ is not a clique in $G$, the graph $H$ is not $\mathcal{R}_2$-rigid by Theorem \ref{thm:coning_cof}. Hence $2\deg_G(u)-\rr_2(H)\geq 4$.  Lemma \ref{lem:cof:bounds}(b) now gives 
$$\rc_2^1(G,u)\geq \frac{4+|\sigma^{-1}(u)|}{4} +\sum_{X\in \hXX_u}\Big(\frac{9}{4}-\frac{6}{|X|}\Big)> 1+ \frac{|\sigma^{-1}(u)|}{4}+|\hXX_u|,$$
since $\frac{9}{4}-\frac{6}{|X|}>1$ for all $X\in  \hXX_u$. This verifies
the claim  when $|\hXX_u|\geq 2$.

It remains to consider the case $|\hXX_u|\leq 1$. Then, since $|\sigma^{-1}(u)|\geq 1$, it follows from Claim \ref{claim:deficient}(b) that $\hXX_u= \{X\}$ for some $X\in \hXX$ with $|X|=5$ and 
$|\sigma^{-1}(u)|\leq |X|-1= 4$.
Lemma \ref{lem:cof:bounds}(a) and Claim \ref{claim:deficient}(c) now give
$$\rc_2^1(G,u)= \frac{\deg_G(u)}{2} + \Big(\frac{7}{2}-\frac{6}{5}-\frac{5}{2}\Big) \geq 3+\frac{4}{5} \geq 3+\frac{|\sigma^{-1}(u)|}{5}\,.$$
This completes the proof of the claim.
\end{proof}

By Claims \ref{claim:deficient}(a) and \ref{claim:surplus}, for each $u\in S$, we have
$$\sum_{v \, \in \, \sigma^{-1}(u)\cup \{u\}} \rc_2^1(G,v)\geq 3\cdot \big|\sigma^{-1}(u)\cup \{u\}\big|.$$
Hence,
\begin{gather*}
\rr_2^1(G)=\sum_{v\,\in\, V}\rc_2^1(G,v)=\sum_{\substack{u\, \in\, S\\v\,\in\, \sigma^{-1}(u)\cup \{u\}}}\rc_2^1(G,v)
\geq \sum_{u\,\in\, S}3\cdot\big|\sigma^{-1}(u)\cup\{u\}\big|=3|V|.
\end{gather*}
This contradicts the bound $\rr_2^1(G)\leq 3|V|-6$ and completes the proof of the theorem.   
\end{proof}

\begin{example}\label{ex:kette}
The following graph $G$ shows that the connectivity hypothesis in Theorem \ref{thm:zeo:general} is best possible.
For an integer $t\geq 5$, let $H$ be a cyclic sequence of $t$
copies of $K_4$ and $t$ copies of $K_5$, alternating around the
cycle, such that consecutive cliques in this cyclic sequence intersect in
a pair of vertices, while nonconsecutive cliques are disjoint. More precisely, with subscripts taken modulo $2t$, let
$$
V(H)=\{u_j,v_j:1\leq j\leq2t\}
\cup\{w_{2i-1}:1\leq i\leq t\},
$$
with
$
A_i=H[\{u_{2i-2},v_{2i-2},u_{2i-1},v_{2i-1}\}]
\cong K_4
$
and
 $B_i=H[\{u_{2i-1},v_{2i-1},w_{2i-1},
u_{2i},v_{2i}\}]
\cong K_5
$ for each $1\leq i\leq t$, and let the edge set of $H$ be the union of the edge sets of the graphs $A_i$ and $B_i$. 
Now construct $G$ from the disjoint union of five copies of $H$ by 
adding the edges of a $K_5$ on the five copies of the vertex $w_{2i-1}$ for each $1\leq i\leq t$. Then $G$ is a 5-connected $K_4$-covered graph. The maximal cliques of $G$ consist of $5t$ copies of $K_4$ and $6t$ (pairwise edge-disjoint) copies of $K_5$.
Moreover, $G[N_G(v)]$ is the union of two
complete graphs with one vertex in common, for all $v\in V(G)$.
Thus
$G$ is locally $\cR_2$-closed, and hence
Theorem~\ref{thm:rcdress:ineq} gives
 $$\rr_2^1(G)\leq \cval(G)=|F(G)|+\sum_{X\in \hXX(G)}(3\cdot |X|-6)= 5t\cdot 4+6t\cdot 9=74t.$$
On the other hand, $|V(G)|=5(5t)=25t$. So we have $\rr_2^1(G)<3|V(G)|-6$ whenever $t\geq 7$.
\end{example}

\subsection{$\mathcal{C}_2^1$-rigidity of body-pin frameworks} 
A {\em body-pin graph} is  a graph $G_H$ obtained from a multigraph $H=(W,E)$
by replacing each vertex $w\in W$ by a complete graph $B_w$ on at least $d_H(w)+4$ vertices,  called the {\em body of $w$}, and then, for each edge $e=xy\in E$,  contracting a vertex of $B_x$ and a vertex of $B_y$ into a single {\em pin vertex}
in such a way that the pin vertices  are pairwise distinct in $G_H$.
We will verify the following cofactor version of Conjecture \ref{con:body-pin}.

A conjectured characterisation of body-pin graphs which are rigid in $\R^3$ was
posed independently by the first two authors and Tanigawa in 2009 and 2011, respectively.
It was eventually published in 2019 by Kir\'aly and Tanigawa \cite[Conjecture 5]{KT}. 
The conjecture is     in terms of partitions $\cP$
of the vertex set of the multigraph $H=(W,E)$.
For two disjoint sets $X,Y\subseteq W$ let 
$d_H(X,Y)$ denote the number of edges of $H$ between $X$ and $Y$, and put
$$
\ell_H(X,Y)=
\begin{cases}
6 \mbox{ if $d_H(X,Y)\geq 3$,}\\
5 \mbox{ if $d_H(X,Y)=2$,}\\
3 \mbox{ if $d_H(X,Y)=1$,}\\
0 \mbox{ if $d_H(X,Y)=0$.}
\end{cases}
$$

\begin{conjecture}\label{con:body-pin}
Let $G_H$ be a body-pin graph defined by a multigraph $H=(W,E)$. Then 
$G_H$ is rigid in $\R^3$
if and only if
$$\sum_{1\leq i<j\leq t}
\ell_H(P_i,P_j) \geq 6(|{\cal P}|-1)$$
for all partitions ${\cal P}=\{P_1,P_2,\dots,P_t\}$ of $W$.
\end{conjecture}

\noindent
We shall prove that the following $\cC_2^1$-cofactor version of Conjecture \ref{con:body-pin} holds. 
In view of Conjecture \ref{con:walter}, this gives strong evidence in support of Conjecture \ref{con:body-pin}.

\begin{theorem}\label{thm:body-pin_C12}
Let $G_H$ be a body-pin graph defined by a multigraph $H=(W,E)$. Then 
$G_H$ is $\cC_2^1$-rigid
if and only if
$$\sum_{1\leq i<j\leq t}
\ell_H(P_i,P_j) \geq 6(|{\cal P}|-1)$$
for all partitions ${\cal P}=\{P_1,P_2,\dots,P_t\}$ of $W$.
\end{theorem}

To this end, we define the {\em value} of a partition $\cP=\{P_1,P_2,\dots,P_t\}$ of the vertex set of a multigraph $H$ to be

$$\val_H({\cal P})=6(|{\cal P}|-1) -\sum_{1\leq i<j\leq t} \ell_H(P_i,P_j).$$

Then Theorem \ref{thm:body-pin_C12} tells us that $G_H$ is $\cC_2^1$-rigid if and only if every partition of $W$ has value at most zero. We will show more generally that the maximum value of a partition of $W$ determines the $C_2^1$-rank of $G_H$, $r_2^1(G_H)$. For our purposes, it will be easier to work with the {\em $C_2^1$-degree of freedom of $G_H$} which we define as
$$\dof_2^1(G_H)=3|V(G_H)|-6-r_2^1(G_H).$$
Note that $\dof_2^1(G_H)$ is completely determined by $H$ since changing the size of any of the bodies of $G_H$ will not affect $\dof_2^1(G_H)$. (More generally replacing a $C_2^1$-rigid subgraph of a graph $G$ by a different $C_2^1$-rigid graph with the same vertices of attachment in $G$ will not change $\dof_2^1(G)$.) We will show:

\begin{theorem}\label{thm:BP} Let $H=(W,E)$ be a multigraph, and $G_H$ be a body-pin graph of $H$.  Then $$\dof_2^1(G_H)=\max \{\val_H(\cP):\mbox{$\cP$ is a partition  of $W$}\}.$$
\end{theorem}

We first show that $\val_H({\cal P})$ gives a lower bound on $\dof_2^1(G_H)$ for each partition $\cP$ of $W$. We will need to introduce some new terminology to do this. Given a vertex $w\in W$ we will refer to the vertices of $G_H$ in the body $B_w$ which are not pin vertices as {\em internal vertices of $B_w$}. In addition we will refer to the set of all internal vertices of $B_w$ as the {\em core of $B_w$}. 

\begin{lemma}\label{lem:lb} Let $H=(W,E)$ be a multigraph and ${\cal P}=\{P_1,P_2,\dots,P_t\}$ 
be a partition of $W$. Then $\dof_2^1(G_H) \geq \val_H({\cal P}).$
\end{lemma}

\begin{proof}
Let ${\cal Q}=\{Q_1,Q_2,\dots,Q_t\}$ be a partition of $V(G_H)$ such that, for each part  $P_i$ of $\cP$,
$Q_i$ contains the core of each body $B_w$ with $w\in P_i$,  as well as
the pin vertices which correspond to an edge of $H$ which is induced by $P_i$.
In addition, each pin vertex of $G_H$ which corresponds to an edge of $H$
between two distinct parts $P_i$ and $P_j$ of $\cP$ belongs to either $Q_i$ or $Q_j$.
Construct a graph $G_H^+$ by adding edges to $G_H$ so that each subgraph $G_H[Q_i]$ becomes a clique. Then $\dof_2^1(G_H)\geq\dof_2^1(G_H^+)$ so
it will suffice to show that $\dof_2^1(G_H^+) \geq \val_H({\cal P})$.

For each $1\leq i\leq t$, let $G_i=G_H^+[Q_i]$ and
let $F_i$ be a base of $\cC^1_2(G_i)$. Then $|F_i|=3|Q_i|-6$ for all $1\leq i\leq t$. Since the sets $Q_i$ are pairwise disjoint, $\bigcup_{i=1}^t F_i$ is $\cC^1_2$-independent, and we may choose  a base $F$ of $\cC^1_2(G_H^+)$ which contains
$\bigcup_{i=1}^t F_i$. We claim that 
\begin{equation}\label{eq:harom}
|F\cap E_{G_H^+}(Q_i,Q_j)|\leq \ell_H(P_i,P_j) \mbox{ for all $1\leq i<j\leq t$. }
\end{equation}
This follows since $G_H^+[F_i]$ and $ G_H^+[F_j]$ are both $\cC_2^1$-rigid and  hence $\ell_H(P_i,P_j)$ is the maximum number of edges  we can add between $G_H^+[F_i]$ and $G_H^+[F_j]$, which are incident to the pin vertices for $E_H(P_i,P_j)$, and preserve the $\cC_2^1$-independence of $F_i\cup F_j$.

We can use (\ref{eq:harom}) to deduce that 
\begin{align*}
r_2^1(G_H^+)=|F|=&\sum_{i=1}^t |F_i| + \sum_{1\leq i<j\leq t} |F\cap E_{G_H^+}(Q_i,Q_j)|\\
\leq&\sum_{i=1}^t (3|Q_i|-6) + \sum_{1\leq i<j\leq t} \ell_H(P_i,P_j)\\
=& 3|V(G_H)|-6t + \sum_{1\leq i<j\leq t} \ell_H(P_i,P_j).
\end{align*}
Hence
$$\dof_2^1(G_H)=3|V(G_H)|-6-r_2^1(G_H)\geq 6(|{\cal P}|-1) -\sum_{1\leq i<j\leq t} \ell_G(P_i,P_j)=\val_H({\cal P}).$$
\end{proof}

We will need the
following lemma on the 2-dimensional 
rigidity matroid of the body-pin graph $G_H$ of a multigraph $H$ to verify the reverse inequality.
By a classical result of Tay \cite{tay} on $d$-dimensional body-hinge frameworks, $G_H$ is rigid in $\R^2$ if and only if the multigraph
$2H$ obtained from $H$
by doubling each edge of $H$, contains three edge-disjoint spanning trees. 

\begin{lemma}\label{lem:2dimbodies}
Let $H=(W,E)$ be a multigraph and suppose that, for every subgraph $H'$ of $H$ with at least two vertices, $G_{H'}$ is not rigid in $\R^2$. Then 
$2|E|\leq 3|W|-4$.
\end{lemma}
\begin{proof}
Suppose, for a contradiction, that $2|E|\geq 3|W|-3$.  %
If $|E(H')|\leq 3|V(H')|-3$ for all subgraphs $H'$ of $2H$, then $|E(2H)|=2|E|=3|W|-3$ and $2H$ is the edge-disjoint union of three spanning trees by a theorem of
Nash-Williams \cite{NW}.  This would imply that $G_H$ is rigid in $\R^2$,
by the above mentioned result of Tay, and
contradict the hypothesis of the lemma. Hence we may choose a minimal subgraph  $H'=(W',E')$ of $2H$ such that  $|E'| = 3|W'|-2$. Then $|W'|\geq 2$ and $H'$ contains three edge-disjoint spanning trees. Let $H''=H[W']$. Since $H'$ is a spanning subgraph of $2H''$, $2H''$ also contains three edge-disjoint spanning trees
and  hence $G_{H''}$
is rigid in $\R^2$, a contradiction.
\end{proof}

\subsubsection*{Proof of Theorem \ref{thm:BP}}
We have $\dof_2^1(G_H) \geq \max\{\val_H(\cP):\mbox{ $\cP$ is a partition of $W$}\}$
by Lemma \ref{lem:lb}, so it will suffice to 
show that there exists a partition $\cP$ of $W$ such that $\dof_2^1(G_H) \leq \val_H(\cP)$. We proceed by induction on $|W|$. The case $|W|=1$ is obvious so we may assume that $|W|\geq 2$. 
Since increasing the size of a body does not change the degrees of freedom of $G_H$,
we may also assume that each body of $G_H$ has exactly  $k$ vertices,
for some integer $k\geq 120$. 
(This assumption will simplify some of our counting arguments.)
For each $w\in W$, let $V_w$ and $U_w$ be the sets of vertices and internal vertices, respectively, of the body $B_w$ and put $V=V(G_H)$.
Let $\bar G_H$ denote the $\cC_2^1$-closure of $G_H$. 

Suppose that $V_x\cup V_y$ induces a clique in $\bar G_H$ for two distinct vertices $x,y\in W$, that is, $\dof_2^1(G_H)=\dof_2^1(G_H+K(V_x\cup V_y))$.
Let $H'=(W',E')$ be the multigraph obtained from $H$ by contracting the pair
$\{x,y\}$ into a new vertex $w$.
By induction, there is a 
partition $\mathcal{P'}$ of  $W'$ such that 
$\dof_2^1(G_{H'})\leq \val_{H'}(\mathcal{P'})$. 
Let $\mathcal{P}$ be the %
partition of $W$ obtained from $\mathcal{P'}$ by replacing the part $P'$  of $\mathcal{P}'$ that contains $w$ by $(P'-w)\cup \{x,y\}$. Then, $\val_{H}(\mathcal{P})= \val_{H'}(\mathcal{P}').$ Moreover, we have $\dof_2^1(G_H+K(V_x\cup V_y))=\dof_2^1(G_{H'})$. Thus we obtain 
$$\dof_2^1(G_H)=\dof_2^1(G_H+K(V_x\cup V_y))= \dof_2^1(G_{H'})\leq \val_{H'}(\mathcal{P}')= \val_{H}(\mathcal{P})
,$$ as required.
Hence, we may assume that
\begin{equation}\label{eq:bodyp:cluster}
\mbox{ $V_x\cup V_y $ does not induce a clique in $ \bar G_H$  for all distinct  $x,y\in W$. } 
\end{equation} 
This implies, in particular, that $H$ has no edges of multiplicity greater than two. We will refer to a pair of pin vertices $\{x,y\}$ of $G_H$ which correspond to a pair of parallel edges of  $H$ as a {\em hinge pair} of $G_H$. Let $h(H)$ denote the number of edges of $H$ of multiplicity two, i.e. the number of hinge pairs in $G_H$.

We next define a function $f:V\to \R$.
For $v\in V$, put
$$
f(v)=
\begin{cases}
3-\frac{6}{k}&  \mbox{if $v$ is not a pin vertex of $G_H$,} \\ 
6-\frac{12}{k}& \mbox{if  $v$ is a pin vertex  which belongs to no hinge pair of $G_H$},\\
5.5-\frac{12}{k}& \mbox{if  $v$ is a pin vertex  which belongs to a hinge pair of $G_H$}. 
\end{cases}
$$
A straightforward counting argument tells us that the function $f$ can be used to determine the value of the trivial partition $\cP_0$ of $W$ into $|W|$ singletons: we have
\begin{equation}\label{eq:negy}
\sum_{v\in V}f(v) =\sum_{w\in W}\sum_{v\in V_w}\left(3-\frac{6}{k}\right)-h(H)=(3|V|-6)-\val_H(\cP_0).
\end{equation}   

Since 
$\sum_{v\in V} \rc_2^1(\bar G,v)=\rr_2^1(\bar G)=\rr_2^1(G_H)$ by Lemma \ref{lem:rc_sum},  (\ref{eq:negy}) gives
\begin{equation}\label{eq:hat}
\val_H(\cP_0)+\left(\sum_{v\in V}f(v)\right)\;=\; 3|V|-6\;=\;\dof_2^1(G_H)+\left(\sum_{v\in V}\rc_2^1(\bar G, v)\right).  
\end{equation}
The remainder of the proof will be devoted to showing that 
\begin{equation}\label{eq:het}
\sum_{v\in V}\rc_2^1(\bar G, v)\geq  \sum_{v\in V} f(v).
\end{equation}
Combined with (\ref{eq:hat}), this will imply that the trivial partition $\cP_0$ satisfies the required inequality $\dof_2^1(G_H)\leq \val_H(\cP_0)$.

For $w\in W$, let $\bar B_w$ denote the maximal clique of $\bar G_H$ that contains the body $B_w$, and let $\bar V_w=V(\bar B_w)$. Then  (\ref{eq:bodyp:cluster}) implies that $\bar B_x\neq \bar B_y$ for all distinct 
$x,y\in W$.  Also note that, if $e=xy\in E$ and $v$ is the pin vertex of $G_H$ corresponding to $e$, then $\bar B_x$ and $\bar B_y$ are two distinct maximal cliques of $\bar G_H$ containing $v$. For $v\in V$, we will continue to use the notation from Section \ref{sec:dress} that $\hat \cX_v(\bar G_H)$ denotes the set of maximal cliques of $\bar G_H$ which have at least five vertices and contain $v$.

We say that a vertex $v\in V$ is {\it deficient} if $\rc_2^1(\bar G_H, v)< f(v)$. Let $D$ be the set of all deficient vertices in $V$. We may assume that $D\neq \emptyset$  since otherwise (\ref{eq:het}) holds trivially.

\begin{claim}\label{claim:pin:deficient}
For each  $v\in D$ we have:
\begin{enumerate}[(a)]
\item  %
$\rc_2^1(\bar G_H,v)\geq 5.5-\frac{12}{k}$;            
\item $v$ is a pin vertex of $G_H$ which does not belong to a hinge pair and  $\hXX_v(\bar G_H)=\{\bar V_x,\bar V_y\}$ where  $xy$ is the edge of $H$ corresponding to $v$; 
\item there is a unique vertex $v^*\in V-v$ such that $\bar V_x, \bar V_y\in \hXX_{v^*}(\bar G_H).$

\item $v^*$ is a pin vertex of $G_H$ and there exists a vertex $z\in W\setminus \{x,y\}$ such that $\bar V_z\in \hXX_{v^*}(\bar G_H)$. %
\end{enumerate}
\end{claim}

\begin{proof} We have $\deg_{\bar G}(v)\geq k-1>3$, and hence $\rc_2^1(\bar G,v)\geq 3-6/k$ by Lemma \ref{prop:rc:probsum}. The hypothesis that  $v\in D$ and the definition of $f(v)$ now imply that $v$ is a pin vertex of $G_H$. Let $xy$ be the edge of $H$ corresponding to $v$. Then  $\bar B_x,\bar B_y$ are distinct cliques in $\hXX_v(\bar G_H)$ by (\ref{eq:bodyp:cluster}). Put $G_0=\bar G_H[N_{\bar G_H}(v)]$, and let $J\subseteq E(G_0)$ denote the set of edges in $G_0$ that are not contained in a copy of $K_4$ in $G_0$.  
Since $\bar G_H$ is $\cC_2^1$-closed, $G_0$ is not $\cR_2$-rigid by (\ref{eq:bodyp:cluster}) and Theorem \ref{thm:coning_cof}. Thus $\dof_2(G_0-J)\geq 1$.
Lemma \ref{lem:cof:bounds}(b) and the facts that $|\bar V_x|\geq k$, $|\bar V_y|\geq k$ and $|X|\geq 5$ for all  $X\in \hXX_v(\bar G_H)\setminus\{\bar V_x, \bar V_y\}$ now give 
\begin{align}\label{eq:barGv:bound}
\rc_2^1(\bar G_H,v)&= \frac{\dof_2(G_0-J)+3}{4}
+\sum_{X\in \hXX_v(\bar G_H)}\bigg(\frac{9}{4}-\frac{6}{|X|}\bigg)\\
&\geq \frac{\dof_2(G_0-J)-1}{4}+ 5.5-\frac{12}{k}+\frac{21}{20}\cdot \big|\hXX_v(\bar G_H)\setminus\big\{\bar V_x, \bar V_y\big\}\big|.\label{eq:barGv:bound2}
\end{align}
Substituting $\dof_2(G_0-J)\geq 1$ into  \eqref{eq:barGv:bound2} proves (a). Since $v\in D$, \eqref{eq:barGv:bound2} 
also implies (b). In addition, Theorem \ref{thm:CJT:K5} 
and
the fact that $\hXX_v(\bar G_H)=\{\bar V_x,\bar V_y\}$ gives  $$N_{\bar G_H}(v)=(\bar V_x\cup \bar V_y)\setminus\{v\}.$$
Since $\bar G_H$ is $\cC_2^1$-closed, $G_0-J$ is the union of the two complete graphs $\bar B_x-v$ and $\bar B_y-v$. Thus, if $\bar V_x \cap  \bar V_y=\{v\}$, then $\dof_2(G_0-J)= 3$, which contradicts $\rc_2^1(\bar G,v)< f(v)$ and (\ref{eq:barGv:bound2}). Hence, we have $|\bar V_x \cap  \bar V_y|=2,$ which implies (c).

For each $w\in W$, the vertices in $U_{w}$ have the same neighbour set in $G_H$. This, and the fact that $\rr_2^1(G_1)=\rr_2^1(G_2)$ for any two isomorphic graphs $G_1,G_2$, implies  that in each body of $G_H$ all internal vertices have the same neighbour set in $\bar G_H$. The uniqueness of $v^*$ now implies that $v^*$ cannot be an internal vertex of $G_H$. 
Thus, $v^*$ is a pin vertex of $G_H$ corresponding to some edge $zz'$ of $H$. In addition, $zz'$ and $xy$ cannot be parallel edges, since $v$ does not belong to a hinge pair. Hence we may assume that $z\notin \{x,y\}$.
Then $\bar V_z\in \hXX_{v^*}(\bar G_H)$, which proves (d).
\end{proof}

Note that Claim \ref{claim:pin:deficient} implies that $v^*\not\in D$ for each $v\in D$ (since $|\hXX_{v^*}(\bar G_H)|\geq 3$ by (c) and (d), so $v^*$ does not  satisfy (b)). Hence the map $v\mapsto v^*$ defines a function $\sigma:D\to V\setminus D$. 
For each $u\in V\setminus D$, let $\sigma^{-1}(u)=\{v\in D: v^*=u\}$. Let
$S=\sigma(D)$ and $T=V\setminus (D\cup S).$

\begin{claim}\label{claim:pin:surplus}
For each $u\in S$, 
\begin{equation*}
\rc_2^1(\bar G_H,u)\geq 
6-\frac{12}{k}+\frac{|\sigma^{-1}(u)|}{2}.
\end{equation*}
\end{claim}
\begin{proof}
Let $$W_u=\{w\in W: \bar V_w\in \hXX_u(\bar G_H)\}.$$
For every $w\in W$, we have 
${6}/|\bar V_w|\leq {6}/{k}\leq {1}/{20}$,
since $k\geq 120$. Since $u\in S$, we have $u=v^*$ for some $v\in D$ and Claim \ref{claim:pin:deficient}(c),(d) implies that $|W_u|\geq 3$. Hence $\bar G[N_{\bar G}(u)]$ is not $\cR_2$-rigid by (\ref{eq:bodyp:cluster}) and Theorem \ref{thm:coning_cof}. We can now use Lemma \ref{lem:cof:bounds}(b) to deduce that
\begin{equation}\label{eq:bpin:surpl:bound}
\rc_2^1(\bar G_H, u)\geq 1+\frac{22}{10}\cdot |W_u|. 			\end{equation}
If $|\sigma^{-1}(u)|\leq 3$, then  (\ref{eq:bpin:surpl:bound}) and the fact that $|W_u|\geq 3$ give $$\rc_2^1(\bar G_H, u)\geq \frac{76}{10} >  6-\frac{12}{k} + \frac{|\sigma^{-1}(u)|}{2},$$ as required. Hence we may assume that 
$|\sigma^{-1}(u)|\geq 4$. We will show that the claim holds in this case by applying Lemma \ref{lem:2dimbodies} to a subgraph of $H[W_u]$.

For each edge $e$ of $H$, let $v_e$ be the pin vertex of $G_H$ corresponding to $e$.
Let $E_u=\{e\in E:v_e\in \sigma^{-1}(u)\}$, let $W'_u$ be the set of vertices of $H$ incident to the edges in $E_u$ and  put  $H_u=H[W'_u]$. Note that if $w\in W_u'$ then $w$ is incident with an edge $e\in E_u$. This implies that $v_e\in D$ and $v_e^*=u$. We can now use Claim \ref{claim:pin:deficient}(c) to deduce that $w\in W_u$ and hence $W_u'\subseteq W_u$.

Let
$V'_u=\bigcup_{w\in W'_u} (V_w-u)$ and put $G_u=G_H[V'_u]$.
Then $G_u$ is the body-pin graph of $H_u$. 
In addition, $V'_u\subseteq N_{\bar G_H}(u)$.
This implies that, for any pair  of distinct vertices $x,y\in W_u'$, $(V_x-u)\cup (V_y-u)$ is not contained in an $\mathcal{R}_2$-rigid subgraph of $G_u$; otherwise, $V_x\cup V_y$ would be contained in a $\cC_2^1$-rigid subgraph of $\bar G_H$ by Theorem \ref{thm:coning_cof}, which would contradict (\ref{eq:bodyp:cluster}). We can now apply
Lemma \ref{lem:2dimbodies} to $H_u$ to deduce that
\begin{equation}\label{eq:A0F0:bound}
3|W'_u|-4\geq 2|E_u| =2|\sigma^{-1}(u)|.
\end{equation}
By using \eqref{eq:bpin:surpl:bound} and \eqref{eq:A0F0:bound}, and the fact that $W_u'\subseteq W_u$, we  obtain
$$\rc_2^1(\bar G_H, u)\geq 1+\frac{22}{10}\cdot |W'_u|\geq 1+\frac{22}{10}\cdot\frac{2 |\sigma^{-1}(u)|+4}{3}\geq 6+\frac{|\sigma^{-1}(u)|}{2},$$
where the 
last inequality follows since $|\sigma^{-1}(u)|\geq 4$. This completes the proof of the claim.
\end{proof}

By Claims \ref{claim:pin:deficient}(a) and 
\ref{claim:pin:surplus} and the definition of $f$, for every $u\in S$, we have
$$	\sum_{v \, \in \,\sigma^{-1}(u) \cup \{u\}} \rc_2^1(\bar G,v) \geq 
6-\tfrac{12}{k} +\tfrac{|\sigma^{-1}(u)|}{2}+\sum_{v \, \in \, \sigma^{-1}(u)}(5.5-\tfrac{12}{k})
\geq \sum_{v \, \in \,\sigma^{-1}(u) \cup \{u\}}f(v).$$
Since $T\cap D=\emptyset$, this gives
$$\sum_{v\in V}\hspace{0cm}\rc_2^1(\bar G, v)=\hspace{-0.4cm} \sum_{\substack{u\, \in\, S\\v\,\in\, \sigma^{-1}(u)\cup \{u\}}}\hspace{-0.4cm}\rc_2^1(\bar G,v) + \sum_{v\in T} \rc_2^1(\bar G,v) \geq
\hspace{-0.4cm}\sum_{\substack{u\, \in\, S\\v\,\in\, \sigma^{-1}(u)\cup \{u\}}}\hspace{-0.4cm}f(v) + \sum_{v\in T} f(v)
= \hspace{-0.2cm}\sum_{v\,\in\, V}\hspace{0cm}f(v),$$
which completes the proof of the theorem.
\qed

\section{Rigidity and global rigidity of highly connected $K_t $-covered graphs in $\mathbb{R}^3$}

We will show that every 5-connected, $K_5$-covered graph is globally rigid in $\R^3$, that every 7-connected, $K_4$-covered graph is rigid in $\R^3$ and that every 5-connected, $K_4$-covered graph with no $\cR_3$-bridges is rigid  in $\R^3$. 

Throughout this section we will only be concerned with the 3-dimensional rigidity matroid, so will simplify notation by suppressing the subscript 3 in our notation for these matroids. For example we will write $\cR(G)$, $\rr(G)$ and $\rc(G,v)$ instead of $\cR_3(G)$, $\rr_3(G)$ and $\rc_3(G,v)$. Similarly we will say that a graph $G$ is (globally) rigid to mean $G$ is (globally) rigid in $\R^3$. 

\subsection{The degree of freedom of a set of vertices in a graph}

Given a graph $G=(V,E)$ and $U\subseteq V$, the {\em degree of freedom of $U$ in $G$} is given by
$$\dof_G(U)=\rr(G+K(U))-\rr(G).$$
In this subsection, we will study the parameter $\dof_{G-v}(N_G(v))$, for a vertex  $v\in V$, and its relation to $\rr(G)-\rr(G-v)$.
We begin with some elementary observations.

\begin{lemma}
\label{dofeasy}
Let $G=(V,E)$ be a graph, let $F$ be a set of edges on the vertex set $V$, let $U,U'\subseteq V$ with $U\subseteq U'$, and let $k$ be a positive integer. Then\\
(a) $\dof_{G+F}(U)+|F|\geq \dof_G(U) \geq \dof_{G+F}(U)$, \\
(b)  $\rr(G+F)\geq \rr(G)+k$ whenever $\dof_G(U) \geq \dof_{G+F}(U) + k$,  and \\
(c) $\dof_G(U')\geq \dof_G(U).$
\end{lemma}

\begin{proof}
Parts (a) and (b) follow from the observation that, for any two graphs $H_1\subseteq H_2$ on $V$,  adding an edge $f\in F$  to $H_2$ increases $r(H_2)$ by  at most one, and if this does increase $r(H_2)$ by one, then adding $f$ to $H_1$ increases $r(H_1)$ by one.  
Part (c) follows from the monotonicity of the rank function.
\end{proof}

\begin{lemma}\label{lem:dof:NGv:bound}
Let $G=(V,E)$ be a graph and let $v\in V$ be  a vertex with $\deg_G(v)\geq 3$. Then
$$\dof_{G-v}(N_G(v))-\dof_{G}(N_G(v))=\rr(G)-\rr(G-v)-3.$$
\end{lemma}
\begin{proof}
By using the definition of the degree of freedom of a subset we obtain
\begin{align*}
&\dof_{G-v}(N_G(v))-\dof_G(N_G(v))\\
&=\bigl[\rr(G-v+K(N_G(v)))-\rr(G-v)\bigr]
 -\bigl[\rr(G+K(N_G(v)))-\rr(G)\bigr]\\
&=\rr(G)-\rr(G-v)-3.
\end{align*}
where the last equality follows from the fact that $v$ is a simplicial vertex in $G+K(N_G(v))$.
\end{proof}

The next lemma follows  from the 3-dimensional case of 
Lemma \ref{lem:1ext}(a).  

\begin{lemma}\label{lem:01ext:dof}
Let $G=(V,E)$ be a graph.
Let $k\in \{0,1\}$ and $v\in V$ with $\deg_G(v)\geq 3+k$. If $\dof_{G-v}(N_G(v))\geq k$, then
$\rr(G)-\rr(G-v)\geq 3+k$. 
\end{lemma}

Note that the extension of Lemma \ref{lem:01ext:dof} to the case when $k=2$ does not hold in general:
consider, for example, the graph obtained by gluing two complete graphs $K_5$ and $K_2$ along a vertex $v$. We shall show in Lemma \ref{lem:2ext:unified}  however that this extension does hold when the neighbour set of $v$ in $G$ is the union of two distinct cliques of cardinality at least three. We first establish a preliminary lemma.

\begin{lemma}\label{lem:veebridge}
Let $G=(V,E)$ be a graph and $V_1, V_2$ be the vertex sets of two  disjoint cliques in $G$ with 
$|V_i|\geq 3$ for $i=1,2$. Let $D=\{xy\in \cl(G): x\in V_1,y\in V_2\}$. Suppose that $|D|\leq 2$ and 
$\dof_{G}(V_1\cup V_2)\geq 2$.
Then there exist two adjacent edges $e,f$ joining  $V_1$ and $V_2$ such that $\rr(G+e+f)=\rr(G)+2$.
\end{lemma}

\begin{proof}
Since $|V_1|\geq 3$ and $|D|\leq 2$, there is a vertex $x\in V_1$ such that  $xv\not\in \cl(G)$ for all $v\in V_2$. Let $a$ be an arbitrary vertex in $V_2$. Since $\rr(G+K(V_1\cup V_2))\geq \rr(G)+2$, there exist vertices $y\in V_1$ and  $b\in V_2$ such that $\rr(G+xa+yb)= \rr(G)+2$. If $x=y$ or $a=b$, then we are done, so we may assume that this is not the case. Put $e=xb$. Then $e\not\in \cl(G)$ by the choice of $x$ and $\{xa,yb\}\not\subseteq \cl(G+e)$ since $\rr(G+xa+yb)\geq \rr(G)+2$. Hence 
$\rr(G+e+f)=\rr(G)+2$ for some $f\in \{xa,yb\}$, and the lemma holds. 
\end{proof}

\begin{lemma}\label{lem:2ext:unified}
Let $G=(V,E)$ be a graph and $v\in V$ with $\deg_G(v)\geq 5$. 
Let $V_1,V_2$ be the vertex sets of two distinct cliques of $G[N_G(v)]$ with $|V_i|\geq 3$ for $i=1,2$, and suppose that $\dof_{G-v}(V_1\cup V_2)\geq 2$. Then
$\rr(G)-\rr(G-v)\geq 5$.
\end{lemma}

\begin{proof}
Let $F=E(\cl(G)-v)-E(G-v)$. 
If 
$\rr(G+F-v)\geq \rr(G-v)+2$, then the 0-extension property (c.f. Lemma \ref{lem:1ext})
implies that $\rr(G)\geq \rr(G+F-v)+3\geq \rr(G-v)+5$, as required. Hence we may assume that $\rr(G+F-v)\leq \rr(G-v)+1$. Similarly, if
$\rr(G+F-v)=\rr(G-v)+1$, then Lemma \ref{dofeasy}(b)  and the hypothesis that $\dof_{G-v}(V_1\cup V_2)\geq 2$ imply that $\dof_{G+F-v}(V_1\cup V_2)\geq 1$,
and the 1-extension property (c.f. Lemma \ref{lem:1ext})
implies that $\rr(G)\geq \rr(G+F-v)+4\geq \rr(G-v)+5$.
Hence we may assume that $\rr(G+F-v)=\rr(G-v)$. This implies, again by Lemma \ref{dofeasy}(b), that
$\dof_{\cl(G)-v}(V_1\cup V_2)=\dof_{G-v}(V_1\cup V_2)$. Replacing $G$ by $G+F$ if necessary, we may assume that $G-v=\cl(G)-v$. By Lemma \ref{dofeasy}(c), we may also assume that $V_1$ and $V_2$ induce maximal cliques in $G[N_G(v)]$.
Since  $\dof_{G-v}(V_1\cup V_2)\geq 2$, we have
$|V_1\cap V_2|\leq 1$. 

First suppose that $|V_1\cap V_2|=  1$ and let $V_1\cap V_2=\{w\}$. 
Since $\dof_{G-v}(V_1\cup V_2)\geq 2$, we may choose a pair $e,f$ of edges between $V_1-\{w\}$ and $V_2-\{w\}$ such that $\rr(G-v+e+f)= \rr(G-v)+2$. We can now apply Lemma \ref{lem:coplanar} (when $e,f$ are disjoint) or Lemma  \ref{lem:2ext:split} (when $e,f$ have a common vertex) to deduce that $\rr(G)\geq  \rr(G-v)+5$.

Next suppose that $|V_1\cap V_2|=0$. 
Let $D=\{xy\in E: x\in V_1, y\in V_2\}$. Since $G-v=\cl(G)-v$ and $V_1,V_2$ are maximal cliques in $G[N_G(v)]$, we have $|D|\leq 2$.
By applying  Lemma~\ref{lem:veebridge} to $G-v$ we can now deduce  that there exist two adjacent edges $e,f$ between $V_1$ and $V_2$ such that $\rr(G-v+e+f)\geq \rr(G-v)+2$.
Then Lemma \ref{lem:2ext:split} gives  $\rr(G)\geq \rr(G-v)+5$, as required.
\end{proof}

\begin{lemma}\label{lem:useful}
Let $G=(V,E)$ be a graph, $v\in V$ and $V_1,V_2$ be the vertex sets of two distinct  cliques of $G[N_G(v)]$ with $|V_i|\geq 3$ for $i=1,2$. Suppose that $\dof_{G}(V_1\cup V_2)\geq 1$.
Then $\dof_{G-v}(V_1\cup V_2)\geq 3$ and $\rr(G)\geq \rr(G-v)+5$.
\end{lemma}
\begin{proof} 
By deleting edges incident with $v$, if necessary,
we may assume that $N_G(v)=V_1\cup V_2$. 
Let $k=\dof_{G-v}(V_1\cup V_2)$ and suppose, for a contradiction,  that $k\leq 2$. Then
Lemmas \ref{lem:dof:NGv:bound}, \ref{lem:01ext:dof}, and \ref{lem:2ext:unified} imply that
$$k-1=			\dof_{G-v}(V_1\cup V_2)-1
\geq \dof_{G-v}(N_G(v))-\dof_{G}(N_G(v))
=  r(G)-r(G-v)-3 \geq k,$$
a contradiction.
Thus $\dof_{G-v}(V_1\cup V_2)\geq 3$ and hence, by Lemma 
\ref{lem:2ext:unified}, we have $\rr(G)\geq \rr(G-v)+5$.
\end{proof}

\subsection{Rank contributions of vertices in $K_5$-covered graphs}

We will obtain
lower bounds on the rank contribution of a
vertex in a $K_5$-covered graph. Our results immediately imply that every 5-connected $K_5$-covered graph is rigid in $\R^3$. Extensions of this result to global rigidity and to $K_4$-covered graphs are given in Sections \ref{sec:globrigidK5} and \ref{subsec:7conn}, respectively.

Given a graph $G=(V,E)$ and $U\subseteq V$ we say that $U$ is a {\em rigid cluster} of $G$ if $\dof_G(U)=0$, or equivalently, if $U$ induces a clique in $\cl(G)$.

\begin{lemma}\label{lem:rc:twocliques}
Let $G=(V,E)$ be a graph and $v\in V$ such that $G-v=\cl(G) -v$.
Suppose that $N_G(v)=V_1\cup V_2$, where
$|V_i|=4$ and $G[V_i]$ is a  maximal clique of $G[N_G(v)]$  for $i=1,2$.
Then $\rc(G,v)\geq 3$.
\end{lemma}

\begin{proof} %
Let $\pi$ be a uniformly random ordering of $V$.
Lemma \ref{prop:rc:probsum} gives
\begin{equation}\label{eqnew:100}
\rc(G,v)\geq 3-\frac{6}{\deg_G(v)+1}+\sum_{i=4}^5\Prob(\rc(G,v,\pi)\geq i).
\end{equation}
Our goal is to obtain lower bounds on the probabilities of the events $\rc(G,v,\pi)\geq 4$ and $\rc(G,v,\pi)\geq 5$. Recall that 
$T_v^\pi$ denotes the set of all vertices  which precede $v$ in the order $\pi$ and $N_v^\pi$ is the set of neighbours of $v$ in $T_v^\pi$.
Let $Q$ be the event that $|N_v^\pi|\geq 4$ and $N_v^\pi\neq V_i$ for $i=1,2$.
If $Q$ occurs, then $N_v^\pi$ does not induce a clique in $G$; otherwise we would contradict the maximality of $V_1$ or $V_2$.
We can now use the hypothesis that $G-v=\cl(G)-v$ and
Lemma \ref{lem:1ext}(a) to deduce  that $\rc(G,v,\pi)\geq 4$ whenever $Q$ occurs.
Let $Q'$ be the event that $|N_v^\pi\cap V_i|\geq 3$ for $i=1,2$. 
We may apply Lemma \ref{lem:useful} to $G[T_v^\pi\cup \{v\}]$ to deduce that $\rc(G,v,\pi)\geq 5$ whenever $Q'$ occurs.

The fact that $N_G(v)$ is not a rigid cluster in $G$ implies that $|V_1\cap V_2|\in \{0,1\}.$ We consider these two possibilities separately.

\medskip
\noindent {\bf Case 1.} $|V_1\cap V_2|=1$. 
Then $|V_1\cup V_2|=7$ and, since $\pi$ is a uniformly random ordering, $v$ is equally likely to appear at any position in the ordering $\pi$ induces on the eight vertices in $N_G(v)\cup\{v\}.$
Hence
\begin{equation}\label{fq2}
\Prob(\rc(G, v, \pi)\geq 4) \geq \Prob(Q)=\Prob\big(|N_v^{\pi}|\geq 4\big)- \sum_{i=1}^2\Prob\big(N_v^{\pi}= V_i\big)
= \frac{4}{8}-\frac{1}{8}\cdot \frac{2}{\binom{7}{4}},
\end{equation}
and
\begin{equation}\label{fq3}
\Prob(\rc(G, v, \pi)\geq 5) \geq \Prob(Q')= \sum_{i=5}^7
\frac{1}{8}\cdot\Prob\big(Q'\,\big|\,|N_v^{\pi}|=i\big) =\frac{1}{8}\cdot\frac{9}{\binom{7}{5}}+\frac{2}{8}.
\end{equation}
\Cref{eqnew:100,fq2,fq3}  and 
$\deg_G(v)= 7$
now give
\begin{align*}
\rc(G,v)\geq 3- \frac{6}{8}+\frac{4}{8}+\frac{2}{8}+\frac{1}{8}\left(\frac{9}{\binom{7}{5}}-\frac{2}{\binom{7}{4}}\right)> 3.
\end{align*}

\medskip
\noindent {\bf Case 2.} $|V_1\cap V_2|=0$.
\medskip
We proceed as in Case 1.
    We have $|V_1\cup V_2|=8$, 
    \begin{equation}\label{fq5}
        \Prob(\rc(G, v, \pi)\geq 4) \geq\Prob(Q)= \frac{5}{9}-\frac{1}{9}\cdot \frac{2}{\binom{8}{4}},
    \end{equation}
    and
    \begin{equation}\label{fq6}
        \Prob(\rc(G, v, \pi)\geq 5) \geq \Prob(Q')\geq \Prob\big(|N_v^\pi|\geq 7)=\frac{2}{9}.
    \end{equation}
    \Cref{eqnew:100,fq5,fq6} together with 
    $\deg_G(v)= 8$
    give
    \begin{equation*}
        \rc(G,v)\geq 3 -\frac{6}{9}+\frac{5}{9} +\frac{2}{9}-\frac{1}{9}\cdot \frac{2}{\binom{8}{4}}> 3. \qedhere
    \end{equation*}
\end{proof}

\begin{lemma}\label{lem:mainC5lemma}
    Let $G=(V,E)$ be a graph and $v\in V$ with $\deg_G(v)=k\geq 4$.
    Suppose that  $G-v+K(N_G(v))$ is rigid but $G$ is not. If each edge incident with $v$ belongs to a copy of $K_5$ in $G$, then
    $\rc(G,v)\geq 3$.
\end{lemma}
\begin{proof}

    By 
    Lemma \ref{lem:new_edge}(a), we may assume that 
    $G-v=\cl(G)-v$.
    Since  $\deg_G(v)\geq 4$,
    $G-v+K(N_G(v))$ is rigid, and $G$ is not rigid, it follows that $N_G(v)$ is not a rigid cluster in $G$. 
    Together with $G-v=\cl(G)-v$, this implies that
    there exist $w_1,w_2\in N_G(v)$ such that $w_1w_2\notin E$.  
    Since each edge incident with $v$ belongs to a copy of $K_5$,
    $G$ contains two distinct maximal cliques $C_i=(V_i',E_i)$ with $vw_i\in E_i$
    and $|V_i'|\geq 5$ for $i=1,2$.
    We have $|V_1'\cap V_2'|\leq 2$ since $\cl(G)-v=G-v$ and $w_1w_2\not\in E$.
    Let $V_i\subset V_i'-\{v\}$ with $|V_i|=4$ and $(V_1'\cap V_2')-\{v\}\subseteq V_i$ for both  $i=1,2$. Let $G^-$ be the graph obtained from $G$ by deleting each edge $vz$ with $z\in N_G(v)-(V_1\cup V_2)$. Then $G^-$ satisfies the hypotheses of Lemma \ref{lem:rc:twocliques}, and hence $\rc(G^-,v)\geq 3$ holds. 
     This gives $\rc(G,v)\geq \rc(G^-,v)\geq 3$ by Lemma \ref{lem:new_edge}.
\end{proof}

Lemma \ref{lem:mainC5lemma} implies:

\begin{theorem}\label{coro:c5c5:rigid}
    Every 5-connected $K_5$-covered graph 
    is rigid.
\end{theorem}
\begin{proof}
    Suppose, for a contradiction, that $G=(V,E)$ is a 5-connected $K_5$-covered graph
    but $G$ is not rigid.
    We may
    assume that $G$ has been chosen such that $|V|$ is as small as possible.
    For each $v\in V$, the graph $G-v+K(N_G(v))$ is a 5-connected graph in which each edge belongs to a $K_5$. Hence $G-v+K(N_G(v))$ is rigid by the minimality of $|V|$. We can now apply Lemma \ref{lem:mainC5lemma} to deduce that $\rc(G,v)\geq 3$ for each $v\in V$. Lemma \ref{lem:rc_sum} now gives
    $\rr(G)\geq 3|V|$, a contradiction.
\end{proof}

\begin{example} \label{ex:egy}
We can modify the construction in Example \ref{ex:kette} to show that the connectivity bound in Theorem \ref{coro:c5c5:rigid} 
is best possible. 
Consider the graph $H$ which is a cycle of $K_5$'s
$H_1,H_2,\ldots,H_{2t}$ with $V(H_i)=\{u_i,v_i,w_i,u_{i+1},v_{i+1}\}$,
$1\leq i\leq 2t-1$, $V(H_{2t})=\{u_{2t},v_{2t},w_{2t},u_{1},v_{1}\}$. 
Now construct $G$ from the disjoint union of five copies of $H$ by adding the edges of a $K_5$ on the five copies of the vertex $w_{2i-1}$ for each $1\leq i\leq t$. Then $G$ is essentially 5-connected for $t\geq 5$ (where a graph is  {\em essentially 5-connected} if it is 4-connected and every vertex separator of size four is the neighbour set of a vertex of degree four).
To see that $G$ is not rigid, consider the maximal clique cover $\cX$ of $G$. This consists of $11t$ copies of $K_5$ and has $10t$ hinges of degree two.
In addition, the hinge pairs induce an independent set of edges in $\cR_3(G)$, so can be extended to a basis $B$ of $\cR_3(G)$. Since $B$ can contain at most 9 edges from each copy of $K_5$ this gives 
$\rr(G)=|B|\leq  9(11t)-10t=89t$. On the other hand, $|V(G)|=5(6t)=30t$. So we have $\rr(G)<3|V(G)|-6$ when $t\geq 7$.
\end{example}

We close this subsection with  a `redundant rigidity' result on  5-connected $K_5$-covered graphs which we will need  in the next subsection to show that  every 5-connected $K_5$-covered graph 
    is globally rigid. 
    
    Given a graph  $G=(V,E)$, a {\em rigid cluster} of $G$ is a subset $U\subseteq V$ which induces a clique in $\cl(G)$, or equivalently  a subset $U$ satisfying $\dof_G(U)=0$. We say that $G$ is 
    {\it clique bipartitionable} if
there is a partition $V=V_1\cup V_2$ such that $G[V_i]$ is a clique
for both $i=1,2$. Note that if $G$ is clique bipartitionable, $|V|\geq 4$,
and $G$ has no isolated vertices, then $G$ has a  clique bipartition $V=V_1\cup V_2$
with $|V_i|\geq 2$ for $i=1,2$.

\begin{lemma}\label{lem:canbecovered}
    Let $H$ be a 5-connected $K_5$-covered graph  and let $w$ be a vertex of $H$. 
    Suppose that  $H[N_{H}(w)]$
    is clique bipartitionable.
    Then $H-w$ is rigid.
\end{lemma}

\begin{proof}
    Suppose, for a contradiction, that the lemma is false. Let $H$ be a counterexample with as few vertices as possible and, subject to this condition, as many edges as possible.
    Let $G=(V,E)=H-w$. Then $G$ is not rigid and
    the maximality of $|E(H)|$ implies that
    \begin{equation}\label{eq:rigcluster}
        \mbox{each rigid cluster of $G$ of size at least five induces a clique in $G$.}
    \end{equation}
    For each vertex $v\in V$ let $H_v=H-v+K(N_{H}(v))$. Then $H_v$ is a 5-connected $K_5$-covered graph, for which $H_v[N_{H_v}(w)]$ is clique bipartitionable.
    Thus the minimality  of $|V(H)|$ implies that  $H_v-w=G-v+K(N_G(v))$ is rigid.  
    We will complete the proof by calculating the rank contributions of the vertices of $G$.
    
    We first consider a vertex $v\in V-N_{H}(w)$.
    Since $H$ is a $K_5$-covered graph, 
    each edge incident with $v$ belongs to a copy of $K_5$ in $G$.
    Lemma \ref{lem:mainC5lemma} now implies that 
    \begin{equation}
        \label{out}
        \mbox{$\rc(G,v)\geq 3$   for each $v\in V-N_{H}(w)$.}
    \end{equation}
    
It remains to consider a vertex  $v\in N_H(w)$. 
    Since $H[N_{H}(w)]$
    is clique bipartitionable, and $H$ is a $5$-connected $K_5$-covered graph,
    there is a bipartition 
    $N_H(w)=V_1\cup V_2$, such that $|V_1|,|V_2|\geq 2$,
    and $G[V_1]$ and $G[V_2]$ are both cliques.
    By relabelling, if necessary, we may assume that $v\in V_1$.
    Since $G-v+K(N_G(v))$ is rigid, $G$ is not rigid, and $\deg_G(v)\geq 4$,
    it follows that
    $K(N_G(v))\not\subseteq G$. 
    Let $V_1'$ be a maximal subset of $V$ with $V_1\subseteq V_1'$ for which
    $G[V_1']$ is a clique and let $x\in N_G(v)-V_1'$.
    Since $H$ is a $K_5$-covered graph, 
    there exists a set $V'\subseteq V$ such that $V'$ induces a maximal clique in $G$,  $|V'|\geq 4$, and  $v,x\in V'$.
    We can now use \eqref{eq:rigcluster} to deduce that $|V_1\cap V'|\leq 2$.
    Thus $|V'-V_1|\geq 2$, and hence 
    $\deg_G(v)\geq |V_1|+1$. 
By applying
    Corollary \ref{coro:rc:geq:d}
    and using that $|V_1|\geq 2$ we can now deduce that
$\rc(G,v)\geq \left(3+\frac{1}{2}-\frac{6}{|V_1|+1}\right)$ and hence
$$\sum_{v\in V_1}\rc(G,v)\geq |V_1|\left(3+\frac{1}{2}-\frac{6}{|V_1|+1}\right)\geq 3|V_1|+\left(\frac{|V_1|}{2}-\frac{6|V_1|}{|V_1|+1}\right) \geq 3|V_1|-3.$$
    By symmetry, we also have $\sum_{v\in V_2}\rc(G,v)\geq 3|V_2|-3.$
    We can now apply Lemma \ref{lem:rc_sum} and (\ref{out}) to obtain
    \begin{align*}
        \rr(G)=\sum _{v\in V}\rc(G,v)&= \sum_{v\in V\setminus (V_1\cup V_2)}\rc(G,v)+\sum_{v\in V_1}\rc(G,v)+\sum_{v\in V_2}\rc(G,v)\\
        &\geq  3|V\setminus (V_1\cup V_2)| + \left(3|V_1|-3\right)+\left(3|V_2|-3\right)= 3|V|-6.
    \end{align*}
    This contradicts the assumption that $G$ is not rigid.
\end{proof}

\subsection{Global rigidity of $K_5$-covered graphs}\label{sec:globrigidK5}

We will extend Theorem \ref{coro:c5c5:rigid} by showing that every 5-connected $K_5$-covered graph is in fact globally rigid. 
We  use the following three results, concerning a local version
of global rigidity. These results hold in all dimensions $d$ but to be consistent with our assumption that $d=3$ throughout this section, we  only state these special cases.
Let $x,y$ be a pair of vertices in graph $G$.
We say that $\{x,y\}$ is {\em weakly globally linked in $G$} if
there exists a 
generic realisation $(G,p)$ of $G$ in $\R^3$, for which we have $\|p(x)-p(y)\|=\|q(x)-q(y)\|$ for all equivalent realisations $(G,q)$ of $G$ in $\R^3$.

\begin{lemma}\cite{JVwgl}
    \label{addwgl}
    Suppose that $G$
    is not globally rigid 
    and let $\{x,y\}$ be weakly globally
    linked for some $x,y\in V(G)$. Then $G+xy$ is not globally rigid.
\end{lemma}

\begin{lemma}\cite{JVwgl}\label{lem:wgl_path}
    Let $G=(V,E)$ be a graph, $X\subset V$ and $u,v\in X$ such that %
                           $\{u,v\}$ is $\cR$-linked in $G[X]$.
If there exists a $uv$-path in $G$ that is internally disjoint from $X$, then $\{u,v\}$ is weakly globally linked in $G$.
\end{lemma}

The next lemma is a corollary of Lemma \ref{lem:wgl_path}.

\begin{lemma}\cite{JVwgl,Tani}
    \label{lem:tanigawa}
    Let $G=(V,E)$ be a graph and let $v\in V$ with $\deg_G(v)\geq 4$.
    Suppose that $G-v$ is rigid and $G-v+K(N_G(v))$ is globally rigid.
    Then $G$ is globally rigid.
\end{lemma}

A set $X$ of vertices is said to be a {\it weakly globally linked cluster}
in $G$ 
if $\{x,y\}$ is weakly globally linked in $G$ 
for all $x,y\in X$.

We can now prove the main result of this section. 

\begin{theorem}\label{thm:5c5c:GR}
    Every 5-connected $K_5$-covered graph is globally rigid. 
\end{theorem}

\begin{proof}
    Suppose, for a contradiction, that the theorem is false and 
    let $G=(V,E)$ be a counter-example, 
    with $|V|$ as small as possible, and subject to this condition, $|E|$ as
    large as possible. Since $K_6$ is globally rigid,
    we have $|V|\geq 7$.
    
    It follows from Lemma \ref{addwgl}
    that, for every weakly globally linked pair $\{x,y\}$ of $G$,
    the graph $G+xy$
    is not
    globally rigid. Hence
    the maximality of $|E|$ implies that
    \begin{equation}\label{eq:wglclosed}
        \mbox{
            each 
            weakly globally linked cluster of $G$ of size at least five induces a clique in $G$.
        }
    \end{equation}\

    \noindent
    We will obtain our desired contradiction by showing that the rank contribution of each vertex of $G$ is at least three.
    
    Choose $v\in V$. Since $|V|\geq 7$ and $\deg_G(v)\geq 5$, the graph $G-v+K(N_G(v))$ is a 5-connected $K_5$-covered graph.
    The minimality of $|V|$ now implies that $G-v+K(N_G(v))$ is globally rigid. By using
    Lemma \ref{lem:tanigawa}, we can deduce that $G-v$ is not rigid.
    Hence $N_G(v)$ cannot be 
    covered by the vertex sets of two cliques of $G$ by Lemma \ref{lem:canbecovered}.
    
    \begin{claim}
        \label{claim1}
        Let $C_1=(V_1,E_1)$, $C_2=(V_2,E_2)$ be two maximal cliques in $G[N_G(v)]$.
        Then $|V_1\cap V_2|\leq 1$.
    \end{claim}
    
    \begin{proof}
    Suppose, for a contradiction, that $|V_1\cap V_2|\geq 2$. If	$|V_1\cap V_2|\geq 3$,
        then $G[V_1\cup V_2\cup\{v\}]$ is globally rigid 
        and hence $V_1\cup V_2$ is a weakly globally linked cluster of
        size at least five in $G$. This would contradict the maximality of $C_1,C_2$ by \eqref{eq:wglclosed}. 
        Hence $|V_1\cap V_2|=2$. Then $G[V_1\cup V_2\cup \{v\}]$ is  rigid. By the 5-connectivity of $G$, there exists a path from 
        a vertex $x\in V_1-V_2$ to a vertex $y\in V_2-V_1$ 
        which is internally disjoint from $V_1\cup V_2\cup \{v\}$.
        Then $\{x,y\}$ is weakly globally linked in $G$ by Lemma \ref{lem:wgl_path}.
        It is easy to see that $G[V_1\cup V_2\cup \{v\}]+xy$ is globally rigid. Thus $V_1\cup V_2$ is a weakly globally linked cluster of $G$. 
        This again contradicts the maximality of $C_1,C_2$ by \eqref{eq:wglclosed}, and completes the proof of the claim.
    \end{proof}
    
    Since $G$ is a $K_5$-covered graph, each vertex of $N_G(v)$ belongs to
    a maximal clique of size at least four in $G[N_G(v)]$.
    By Claim \ref{claim1}, we can choose three distinct maximal cliques 
    $C_1,C_2,C_3$ of $G[N_G(v)]$ with $|V(C_i)|\geq 4$.
    Let $V_i=V(C_i)$, $1\leq i\leq 3$.
    
    \begin{claim}
        \label{claim2}
        $|(V_1\cap V_2)\cup (V_1\cap V_3)\cup (V_2\cap V_3)|\leq 2$.
    \end{claim}
    
    \begin{proof}
                Suppose, for a contradiction,  that $|(V_1\cap V_2)\cup (V_1\cap V_3)\cup (V_2\cap V_3)|\geq 3$.
        Then Claim \ref{claim1} implies that there is a vertex set $K=\{z_1,z_2,z_3\}$
        with $V_1\cap V_2=\{z_3\}$, $V_1\cap V_3=\{z_2\}$, and
        $V_2\cap V_3=\{z_1\}$. 
        Then $z_1z_2$ is an edge of $G[V_1\cup V_2\cup \{v\}]$, and 
        hence $G[V_1\cup V_2\cup \{v\}]$ is rigid. 
        In addition, the 5-connectivity of $G$ implies that there exists a path in $G-v-K$ that connects two vertices  $x\in V_1-K$ and $y\in V_2-K$, that is
        internally disjoint from $G[V_1\cup V_2\cup \{v\}]$.
        By Lemma \ref{lem:wgl_path}, $\{x,y\}$ is weakly globally linked in $G$.
        Since $G$ is not globally rigid, Lemma \ref{addwgl}
        implies that $G+xy$ is not globally rigid. 
        On the other hand, the graph $G[V_1\cup V_2\cup \{v\}]+xy$ is globally rigid. (For example, this follows from the fact that $G[V_1\cup V_2]+xy-z_3$ is 2-connected and hence is globally rigid in $\R$, followed by two applications of Theorem \ref{thm:coning}.)
        Hence, $V_1\cup V_2\cup \{v\}$ induces a weakly globally linked cluster in $G+xy$.
        Using Lemma \ref{addwgl} again, we obtain that $G+K(V_1\cup V_2\cup \{v\})$ is still not globally rigid, which contradicts the maximality of $|E|$.
    \end{proof}
    
    Choose $V_i'\subseteq V_i$ with $|V_i'|=4$ for $1\leq i\leq 3$ and
    let $W=V_1'\cup V_2'\cup V_3'$.
    By Claims \ref{claim1} and \ref{claim2}, we have $|W|\geq 10$. Let $\pi$ be a uniformly random ordering of $V$,
    and consider the set $\tilde W = N_v^{\pi}\cap W$.
    
    \begin{claim}
        \label{claim3}
        Suppose $|\tilde W|\geq 5$. Then $\rc(G,v,\pi)\geq 4$. 
    \end{claim}

    \begin{proof} The hypothesis that $|\tilde W|\geq 5$ implies that $N_v^{\pi} \setminus V'_j \not= \emptyset$ for all $1\leq j\leq 3$, and 
         $|N_v^{\pi}\cap V_i'|\geq 2$
        for some $i\in \{1,2,3\}$. 
        The maximality of
        $C_i$ and Claim \ref{claim1} now imply that
        $\tilde W$ does not induce a clique in $G$. 
        It follows from $\eqref{eq:wglclosed}$ that $\tilde W$ does not induce a 
        weakly globally linked cluster in $G$.
        We can now use  Lemma \ref{lem:wgl_path}, together with the fact that every vertex of  $\tilde W$ is adjacent to $v$, to deduce that $\tilde W$ is not a rigid cluster in $G-v$.
        Lemma \ref{lem:1ext} now implies that
        $\rc(G,v,\pi)\geq 4$. This completes the proof of the claim.
    \end{proof}
    
    Claim \ref{claim3} and the fact that $v$ is equally likely to appear in any position in the ordering of $W+v$ induced by $\pi$ gives
    \begin{equation*}
        \Prob(\rc(G,v,\pi)\geq 4)\geq 1-\Prob(|N_v^{\pi}\cap W|\leq 4) = 1-\frac{5}{|W|+1}\geq 1-\frac{5}{11}= \frac{6}{11}.
    \end{equation*}
    
    \noindent
    Lemma \ref{prop:rc:probsum} now gives
    \begin{equation*}
        \rc(G,v) \geq 3-\frac{6}{\deg_G(v)+1} + \Prob(\rc(G,v,\pi)\geq 4) \geq
        3-\frac{6}{11}+ \frac{6}{11} = 3
    \end{equation*}
        By summing over all $v\in V$ and using 
    Lemma \ref{lem:rc_sum}
    we obtain
    $\rr(G)\geq 3|V|$, a contradiction.
\end{proof}

Theorem \ref{thm:5c5c:GR} strengthens a result of \cite{Jsquare} which states that, if $G$ is a graph of minimum degree at least four and
$G^2$ is 5-connected, then $G^2$ is redundantly rigid.

\subsection{Rigidity of 7-connected $K_4$-covered graphs}
\label{subsec:7conn}

We showed in Section \ref{sec:cof}  that $6$-connected $K_4$-covered graphs are $C_2^1$-rigid.
In this subsection we show that a similar result holds for the 3-dimensional rigidity matroid, if we
strengthen the connectivity condition by one.

\begin{theorem}\label{thm:7c4c_rigid}
    Every 7-connected $K_4$-covered graph is rigid.
\end{theorem}

\begin{proof}
    Suppose, for a contradiction, that the theorem is false and let $G=(V,E)$
    be a counter-example with $|V|$ as small as possible and, subject to this, $|E|$
    as large as possible. Thus $G$ is a 7-connected $K_4$-covered graph which is not rigid and hence $|V|\geq 9$.
    Then, by the maximality of $|E|$, 
    \begin{equation}\label{eq:rigcluster:K4}
        \mbox{each rigid cluster of $G$ of size at least four induces a clique in $G$.}
    \end{equation}
    We will obtain the desired contradiction by showing that  $\rc(G,v)\geq 3$ for all
    $v\in V$. 
    
    Choose $v\in V$. 
    The minimality of $|V|$ implies that $N_G(v)$ is not a clique in $G$, since otherwise $G-v$ would be rigid by induction and hence $G$ would be rigid. 
      For $U\subseteq N_G(v)$, let 
      $\tilde G[U]$ denote the graph obtained from $G[U]$ by deleting all the edges that are not contained in a copy of $K_3$ in $G[U]$. Consider the graphs  ${{H}}_i$, $1\leq i\leq 7$,  shown in Figure \ref{fig:gammagraphs}. 
      Our first claim will allow us to split the proof that $\rc(G,v)\geq 3$ into three cases,
     depending on which of the graphs $H_i$ occur as an induced subgraph of  $\tilde G[N_G(v)]$.
    
    \begin{figure}[h]
        \begin{center}
            \includegraphics[width=1\textwidth]{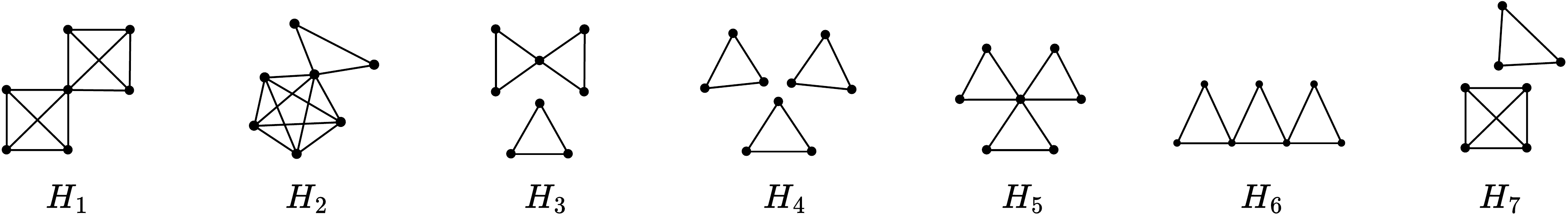}
            \caption{Possible configurations in the neighborhood of a vertex $v$ in the proof of Theorem~\ref{thm:7c4c_rigid}.}\label{fig:gammagraphs}
        \end{center}
    \end{figure}
    
    \begin{claim}\label{claim:gamma}
        There is a set $U\subseteq N_G(v)$ such that 
        $\tilde G[U]$ is isomorphic to ${{H}}_i$ for some $1\leq i\leq 7$.
    \end{claim}
    
    \begin{proof}
        Let $C_1, \dots, C_k$ denote the maximal cliques 
        of size at least three in $N_G(v)$. Since $G$ is $K_4$-covered, $N_G(v)=\bigcup_{i=1}^k V(C_i)$.
         Then $k\geq 2$, since $G[N_G(v)]$ is not a clique.
        Furthermore, Lemma \ref{gluing} and (\ref{eq:rigcluster:K4}) imply that 
        $|V(C_i)\cap V(C_j)|\leq 1$ for all $1\leq i\not= j\leq k$, and that there are no edges in $G$ from $C_i-C_j$
        to $C_j-C_i$ whenever $|V(C_i)\cap V(C_j)|=1$. 
        
        First, suppose that $V(C_i)\cap V(C_j)=\emptyset$ for all $1\leq i<j\leq k$. Then there is a set $U\subseteq N_G(v)$ such that either $\tilde G[U]={{H}}_7$ (when $k=2$), or $\tilde G[U]={{H}}_4$ (when $k\geq 3$).
        
        Next, suppose that there is a pair $1\leq i<j \leq k$ such that $|V(C_i)\cap V(C_j)|=1$. We may assume that $i=1$ and $j=2$. If $k=2$, then there is a set $U$ for which
        $\tilde G[U]\in \{{{H}}_1,{{H}}_2\}$.
        Hence we may assume that $k\geq 3$. Choose $U_i\subseteq V(C_i)$ such that $|U_i|=3$ for $1\leq i\leq 3$, and $V(C_1)\cap V(C_2)\subseteq U_1\cap U_2$. Put $U=U_1\cup U_2\cup U_3$. Then, $\tilde G[U]\in \{{{H}}_3,{{H}}_5, {{H}}_6\}$.
    \end{proof}

    Let $U$ be a fixed subset of $N_G(v)$ such that $H=\tilde G[U]$ is isomorphic to ${{H}}_i$ for some $1\leq i\leq 7$, and  put $R_\pi= N_v^\pi \cap U$ for some uniformly random ordering $\pi$ of $V$. Then
    Lemma \ref{prop:rc:probsum} gives
\begin{equation}\label{eqnew:17}
\rc(G,v)\geq 3-\frac{6}{|U|+1}+\sum_{i=4}^5\Prob(\rc(G,v,\pi)\geq i).
\end{equation}
    Let $Q$ be the event that $|R_\pi|\geq 4$ and $H[R_\pi]$ is not a clique. It follows from
\eqref{eq:rigcluster:K4} and
the fact that $H$ contains all edges which belong to cliques of size three in $G[U]$ that, if $Q$ occurs, then $R_\pi$ is not a rigid cluster in $G$. 
We can now apply  Lemma \ref{lem:1ext}
to deduce that 
\begin{equation}\label{eq:nyolc}
\mbox{$\rc(G,v,\pi)\geq 4$ whenever $Q$ occurs. } 
\end{equation}
    
Let $Q'$ be the event that $H$ contains two maximal cliques $C_1,C_2$ with $|V(C_1)\cap V(C_2)|= 1$ such that $|R_\pi\cap V(C_1)|\geq 3$ and $|R_\pi\cap (V(C_2)-V(C_1))|\geq 2$. We will show that 
\begin{equation}\label{eq:kilenc}
\mbox{$\rc(G,v,\pi)\geq 5$ whenever $Q'$ occurs. } 
\end{equation}
To see this, suppose that  $Q'$ occurs and 
let $G_0$ (resp.\ $G_1$) denote the graph obtained from $G$ by deleting each edge $zv$ with $z\in N_G(v)-R_\pi$ (resp.\ $z\in N_G(v)-(V(C_1) \cup V(C_2))$). 
Under $Q'$ we have $G_1\subseteq \cl(G_0)$ and hence $\rr(G_0)=\rr(G_1)$. Moreover, $V(C_1)\cup V(C_2)$ is not a rigid cluster in $G$ (and in $G_1$) by \eqref{eq:rigcluster:K4} 
and the maximality of $C_1,C_2$. We apply Lemma \ref{lem:useful} to $G_1$ to obtain $\rr(G_0)=\rr(G_1)\geq \rr(G-v) +5$. The submodularity of the rank function $\rr$ now gives 
     $$\rc(G,v,\pi)=\rr(G[T_v^\pi\cup \{v\}])-\rr(G[T_v^\pi])\geq \rr(G_0[T_v^\pi\cup \{v\}])-\rr(G[T_v^\pi])\geq \rr(G_0)-\rr(G-v) \geq 5,$$
     which verifies (\ref{eq:kilenc}).

\begin{claim}\label{claim:egy}
$\rc(G,v)\geq 3$.
\end{claim}
\begin{proof} Consider the following three cases.

    \paragraph{\boldmath Case A: $H={{H}}_i$ for some $i\in \{1,2,3,4,5\}$.}
    We will show that 
    \begin{equation}\label{eq:tiz}
     \Prob(Q)+\Prob(Q')\geq \frac{6}{|U|+1}.   
    \end{equation}
Together with \eqref{eqnew:17}, \eqref{eq:nyolc} and \eqref{eq:kilenc}, this  will give $\rc(G,v)\geq 3$.
    
    We first consider the subcase when $H={{H}}_2$. Let $U_1$ and $U_2$ denote the vertex sets of the copies of $K_5$ and $K_3$ in $H$, respectively. Then,
    \begin{equation}\label{eq10.5}
        \Prob(Q) = \Prob\big(|R_\pi|\geq 4\big)-\Prob\big((|R_\pi|\geq 4) \land (R_\pi\subseteq U_1)\big)
        = \frac{4}{8}-\left(\frac{1}{8}\cdot \frac{5}{\binom{7}{4}}+\frac{1}{8}\cdot \frac{1}{\binom{7}{5}}\right).
    \end{equation}
    Furthermore, $Q'$ occurs if and only if $|R_\pi\cap U_1|\geq 3$ and $|R_\pi\cap (U_2-U_1)|= 2$. Thus, we obtain
    \begin{equation}\label{eq11}
        \Prob(Q')= \sum_{t=5}^7\Prob\big(Q' \land (|R_\pi|=t)\big)= \frac{10}{8\cdot \binom{7}{5}}+ \frac{5}{8\cdot \binom{7}{6}}+\frac{1}{8}.
    \end{equation}
    By combining \eqref{eq10.5} and \eqref{eq11}, we obtain the desired bound $\Prob(Q)+\Prob(Q')= \frac{6}{8}=\frac{6}{|U|+1}.$\smallskip

    In the  remaining four subcases we can use similar, or even simpler calculations to verify \eqref{eq:tiz}. We omit the details but summarize the corresponding values:\medskip
    
    \noindent If $\displaystyle H={{H}}_1$, then 
    $\displaystyle \Prob(Q) = \frac{4}{8}- \frac{2}{8\cdot\binom{7}{4}}$ and 
    $\displaystyle \Prob(Q')=\frac{1}{8}\cdot\frac{15}{\binom{7}{5}}+\frac{2}{8}$.\medskip

    \noindent If $\displaystyle H={{H}}_3$, then 
    $\displaystyle \Prob(Q) = \frac{5}{9}$ and 
    $\displaystyle \Prob(Q') \geq \Prob\big(|R_\pi| = 8\big) = \frac{1}{9}$.\medskip
    
    \noindent If $\displaystyle H={{H}}_4$, then 
    $\displaystyle \Prob(Q) = \frac{6}{10}$ and $\Prob(Q')=0$.	\medskip
    
    \noindent If $\displaystyle H = {{H}}_5$, then 
    $\displaystyle \Prob(Q) = \frac{4}{8}$ and 
    $\displaystyle \Prob(Q') =   \frac{3}{8 \cdot \binom{7}{5}}+\frac{6}{8 \cdot \binom{7}{6}} + \frac{1}{8} = \frac{1}{4}$.

    \paragraph{\boldmath Case B: $H={{H}}_6$.}
    Since $H_6$ contains no clique of size four, we have
    \begin{equation}\label{eq:caseb1}
        \Prob(\rc(G,v,\pi)\geq 4)\geq \Prob(Q)=\Prob(|R_\pi|\geq 4) =\frac{4}{8}=\frac{1}{2}.
    \end{equation}
    It remains to give a lower bound on $\Prob(\rc(G,v,\pi)\geq 5)$. 
    We first note that Lemma \ref{lem:useful} gives 
    \begin{equation}\label{eq:dofGvU}
        \dof_{G-v}(U)\geq \dof_{G-v}(V_1\cup V_2)\geq 3,
    \end{equation}
    where $V_1,V_2\subseteq U$ denote the vertex sets of the two disjoint triangles in $H$.
    We will show that
    \begin{equation}\label{eq:tizenegy}
     \mbox{$\rc(G,v,\pi)\geq 5$ whenever $|R_\pi|=6$.}
     \end{equation}
    To see this, we assume that $|R_\pi|=6$. Then, up to symmetry, there are three cases to consider,
    depending on which vertex of ${{H}}_6$ is not included in $R_\pi$. In all cases except the one depicted in Figure \ref{fig:mount_part}, $\rc(G,v,\pi)\geq 5$ follows directly 
    by applying Lemma \ref{lem:useful} to the graph $G[T_v^\pi\cup \{v\}]$. Hence we may assume that $R_\pi$ is as shown in Figure \ref{fig:mount_part}. Let 
    $e,f$ be the additional solid edges shown in Figure \ref{fig:mount_part} and let $J$ denote the set of additional dotted edges.
    Note that $e$ is not an edge of $\cl(G)$ since otherwise the vertices in the two copies of $K_3$ in $H$ which are incident with $e$ would be a rigid cluster in $G$.
    If all the edges in $ J$ were in $\cl(G-v+e)$ then, since
    $H+e+f+J$ is rigid, $U$ would be a rigid cluster in $G-v+e+f$. This would imply that $\dof_{G-v}(U)\leq 2$, and contradict \eqref{eq:dofGvU}.
    Hence, there is an edge $g\in J$ such that $g$ is not an edge of $\cl(G-v+e)$. Then  $\rr(G-v+e+g)=\rr(G-v)+2$.
    If $g$ is one of the two edges in $J$ that are adjacent to $e$, then we can apply Lemma \ref{lem:2ext:split} to deduce $\rc(G,v,\pi)\geq 5$. If $e$ and $g$ are disjoint, then we use Lemma \ref{lem:coplanar} to deduce $\rc(G,v,\pi)\geq 5$.
    This proves the claim, and hence we obtain 
    \begin{equation}\label{eq:caseb2}
        \Prob(\rc(G,v,\pi)\geq 5)\geq \Prob(|R_\pi|\geq 6)=\frac{1}{4}.
    \end{equation}
    \Cref{eqnew:17,eq:caseb1,eq:caseb2} give
    $$\rc(G,v)\geq \Big(3-\frac{6}{8}\Big)+\frac{1}{2}+\frac{1}{4}= 3.$$

    \begin{figure}[!t]
        \begin{center}
            \includegraphics[scale=0.3]{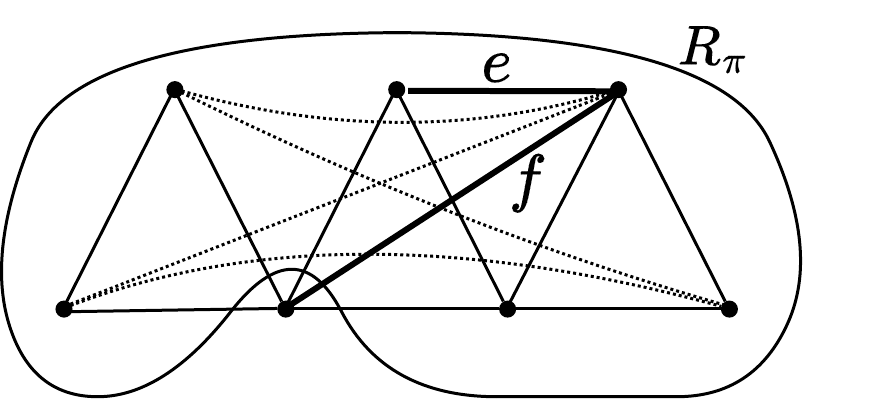}
            \caption{A possible arrangement of $R_\pi$ in Case B of the proof of Theorem \ref{thm:7c4c_rigid}. For this arrangement, Lemma \ref{lem:useful} does not directly imply that $\rc(G,v,\pi)\geq 5.$}\label{fig:mount_part}
        \end{center}
    \end{figure}

    \paragraph{\boldmath Case C: $H={{H}}_7$.} Let $U_1$ and $U_2$ denote the vertex sets of the copies of $K_4$ and $K_3$ in $H$, respectively. %
    Note that if $|R_\pi|\geq 4$ and $R_\pi\neq U_1$, then  $Q$ occurs. Thus, we have
    \begin{equation}\label{eq19}
        \Prob(\rc(G,v,\pi)\geq 4)\geq \Prob(Q)\geq  \frac{4}{8}-\frac{1}{8}\cdot \frac{1}{\binom{7}{4}}.
    \end{equation}
        
        Choose a vertex  $u_1\in U_1$. It follows from \eqref{eq:rigcluster:K4}  that $\dof_{G}(U_2\cup \{u_1\})\geq 1$. Consider the subgraph $G_0$ of $G$ obtained by deleting all edges $vz$ for $z\in N_G(v)-(U_2\cup \{u_1\})$.
        By Lemma \ref{lem:1ext}, we have $\rr(G_0)= \rr(G_0-v)+4$. Lemma \ref{lem:dof:NGv:bound} and the fact that $u_1$ is an arbitrary vertex in $U_1$ now give \begin{equation}\label{eq:220}\dof_{G-v}(U_2\cup \{u_1\})
            \geq 2
             \quad \mbox{for every $u_1\in U_1$.}
        \end{equation}
        
        We say that a pair of adjacent edges $\{e,f\}$ of $K(V)$ is a {\it loose pair in $G-v$} if %
        $\rr(G-v+e+f)=\rr(G-v)+2$.
        Fix a vertex $x\in U_1$. Then (\ref{eq:220}) implies that there exist $a,b\in U_2$ such that $(xa,xb)$ is a loose pair. 
        Let $c\in U_2-\{a,b\}$. By \eqref{eq:rigcluster:K4}, there exists a vertex $y\in U_1-\{x\}$ such that 
        $r(G+cy)=r(G)+1$.
        By (\ref{eq:220}), $U_2\cup \{y\}$ is not a rigid cluster in $G_0-v+yc$. This implies that 
        $\{yc,yd\}$ is a loose pair in $G-v$ for some vertex $d\in \{a,b\}$ . By symmetry, we may assume $a=d$.
         We can now apply Lemma \ref{lem:2ext:split} using the loose pairs $\{xa,xb\}$ and $\{ya,yc\}$ to deduce that $\rc(G,v,\pi)\geq 5$  whenever\\
        $\bullet$  $|R_\pi|=7$, or\\
        $\bullet$ $|R_\pi|=6$ and $a\in R_\pi$, or\\
        $\bullet$ $|R_\pi|=5$, $|R_\pi\cap U_1|= 3$ and either $\{x,a,b\}\subseteq R_\pi$ or $\{y,a,c\}\subseteq R_\pi$.\\
        Thus, we obtain 
        \begin{equation}\label{eq:case:c:5}
            \Prob(\rc(G,v,\pi)\geq 5)\geq \left(\frac{1}{8}\right)+\left(\frac{1}{8}\cdot \frac{6}{7}\right)+\left(\frac{1}{8}\cdot\frac{6}{\binom{7}{5}}\right)=\frac{2}{8}+\frac{1}{56}.
        \end{equation}
        A combination of \cref{eqnew:17,eq19,eq:case:c:5} gives 
        $$\rc(G,v)\geq \Big(3-\frac{6}{8}\Big)+\Big( \frac{4}{8}-\frac{1}{8}\cdot \frac{1}{\binom{7}{4}}\Big)+\Big(\frac{2}{8}+\frac{1}{56}\Big)>3.$$
This completes the discussion of Case C and the proof of the claim.                                \end{proof}	

Since $v$ is an arbitrary vertex of $G$, we can sum the inequality in Claim \ref{claim:egy}	over all $v\in V$ and then apply 
    Lemma \ref{lem:rc_sum}
    to obtain 
    $\rr(G)\geq 3|V|$. This final contradiction completes the proof of the theorem.
\end{proof}

The graph $G$ in Example \ref{ex:kette} shows that 5-connectivity is not sufficient to imply that a $K_4$-covered graph is rigid. (The argument used to show that this graph is not $C_2^1$-rigid applies equally well to rigidity in $\R^3$.)
On the other hand, 5-connectivity is sufficient to imply rigidity for 
squares of graphs with minimum degree at least three (which are necessarily $K_4$-covered),
see \cite[Theorem 3.1]{Jsquare}.

\subsection{Rigidity of 5-connected, $\mathcal{R}$-bridgeless, $K_4$-covered graphs} \label{sec:rigidK4}

We show that the connectivity hypothesis of Theorem \ref{thm:7c4c_rigid}  can be weakened to 5-connectivity for  $K_4$-covered graphs $G$ which have no   $\mathcal{R}$-bridges.

A key step in our proof is to obtain a sufficient condition for an $\mathcal{R}$-bridgeless, $K_4$-covered graph to have a simplicial vertex. %
We will need  the following two results to do this. The first is a special case of \cite[Lemma 5.6]{CJT}.
\begin{lemma}\cite{CJT}\label{lem:cyclic:vertex}
    Let $G=(V,E)$ be an $\mathcal{R}$-bridgeless graph. Then there exists a vertex $v\in V$ such that $\rr(G)\leq \rr(G-v)+4$.
\end{lemma}

The {\em $K_4$-closure} of a graph $G$ is the graph  ${\rm cl}^{K_4}(G)$  obtained from ${\rm cl}(G)$ by deleting every edge in $\cl(G)-E(G)$ which is not contained in a copy of $K_4$.  We say 
that $G$ is 
{\it $\mathcal{R}^{K_4}$-closed} if $G={\rm cl}^{K_4}(G)$.

\begin{lemma}\label{lem:3or5}
    Let $G=(V,E)$ be an $\mathcal{R}^{K_4}$-closed, 
    $K_4$-covered graph 
    and let $v$ be a non-isolated vertex of $G$. Then either $\rr(G)\geq \rr(G-v)+5$ or $\rr(G)=\rr(G-v)+3$.
\end{lemma}
\begin{proof}
    Let $v\in V$. If $N_G(v)$ induces a clique in $G$, then $\rr(G)=\rr(G-v)+3$. Hence we may assume that this is not the case. Since $G$ is the union of copies of $K_4$, there exist distinct vertex sets $V_1,V_2\subseteq N_G(v)$ such that $G[V_i]$ is a maximal clique of $G[N_G(v)]$ and $|V_i|\geq 3$ for $i=1,2$. Since $G$ is $\mathcal{R}^{K_4}$-closed, $V_1\cup V_2$ is not a rigid cluster in $G$. Hence, by Lemma \ref{lem:useful}, $\rr(G)\geq \rr(G-v)+5$.
\end{proof}

\begin{corollary}\label{coro:simplicial}
    Let $G=(V,E)$ be an $\mathcal{R}^{K_4}$-closed, $\cR$-bridgeless, 
    $K_4$-covered graph. Then $G$ has a simplicial vertex.
\end{corollary}
\begin{proof}
    We may assume that $G$ has no isolated vertices. Since $G$ is  $\cR$-bridgeless, Lemma \ref{lem:cyclic:vertex} implies there exists a vertex $v\in V$ with $\rr(G)\leq \rr(G-v)+4$. Lemma \ref{lem:3or5} now gives $\rr(G)= \rr(G-v)+3$. We can now use  Lemma \ref{lem:01ext:dof} to deduce that $\rr(G-v+K(N_G(v)))=\rr(G-v)$.  
    The hypothesis that $G$ is $\mathcal{R}^{K_4}$-closed now implies that $N_G(v)$ induces a clique in $G$ and hence $v$ is a simplicial vertex of $G$.
\end{proof}

\begin{theorem}\label{thm:K4cyclic}
    Every 5-connected, $\cR$-bridgeless, $K_4$-covered graph 
    is rigid in $\R^3$.
\end{theorem}

\begin{proof}
    Let $G=(V,E)$ be a 5-connected, $\cR$-bridgeless, $K_4$-covered graph.
    We prove that $G$ is rigid by induction on $|V|$. Let $G'={\rm cl}^{K_4}(G)$. By Corollary \ref{coro:simplicial}, there is a vertex $v$, such that $N_{G'}(v)$ is a clique in $G'$. It follows that $N_G(v)$ is a rigid cluster in $G$. 
    
    Our aim is to apply induction to $H=G-v+K(N_G(v))$. It is easy to see that $H$ is a 5-connected, $K_4$-covered graph. 
    To see that  $H$ is $\cR$-bridgeless, we suppose for a contradiction that $H$ contains an $\cR$-bridge $e=xy$. Let $p$ be a generic realisation of  $G$ in $\R^3$. Then $(H-e,p|_H)$ has a continuous motion $q$ which changes the distance between $x$ and $y$.   
    Since $N_G(v)$ is a clique in $H$, $q$ induces a continuous motion of $(G-e,p)$ which changes  the distance between $x$ and $y$. This contradicts the fact that $e$ is not an $\cR$-bridge in $G$.
    
    Hence $H$ satisfies the hypotheses of the theorem and we may apply induction to deduce that $H=G-v+K(N_G(v))$ is rigid. We can now use the 0-extension property and the fact that $K(N_G(v))\subset \cl(G)$ to deduce that $G$ is also rigid. 
\end{proof}

 Theorem \ref{thm:K4cyclic} immediately implies Theorem \ref{coro:c5c5:rigid}, since $K_5$-covered graphs are necessarily $\cR$-bridgeless, and Example \ref{ex:egy} shows that the connectivity condition in Theorem \ref{thm:K4cyclic} is best possible.

\section{Concluding remarks}

We believe that the connectivity bound in Theorem  \ref{thm:7c4c_rigid} can be reduced by one. 

\begin{conjecture}
\label{con:zeo}
Every $6$-connected $K_4$-covered graph is rigid in $\R^3$. 
\end{conjecture}

\noindent 
Note that the analogous  statement for the $\cC_2^1$-cofactor matroid follows from Theorem \ref{thm:zeo:general}. Hence Conjecture \ref{con:zeo} would follow from  Conjecture \ref{con:walter}. Note also that Example \ref{ex:kette} shows that Conjecture \ref{con:zeo} would become false if we weakened the connectivity hypothesis to 5-connectivity. On the other hand, 
Theorem \ref{thm:K4cyclic} implies that Conjecture \ref{con:zeo} holds for edge-transitive $K_4$-covered graphs (since an edge-transitive 6-connected graph on $n$ vertices has at least $3n$ edges so contains at least one $\cR_3$-circuit and hence, by edge transitivity, is $\cR_3$-bridgeless).

A $d$-dimensional {\it combinatorial zeolite} is the line graph  of a $(d+1)$-regular graph. (Thus every $d$-dimensional combinatorial zeolite is a $K_{d+1}$-covered graph.)
It was shown in \cite{Jzeo} that a $2$-dimensional  combinatorial zeolite is
globally rigid in $\R^2$ if and only if it is 3-connected.
We conjecture that a similar sufficient condition holds in $\R^3$.

\begin{conjecture}
\label{con:zeo2}
Every $6$-connected, 3-dimensional combinatorial zeolite is globally rigid in $\R^3$.    \end{conjecture}

A stronger conjecture, stating that 
every $4$-connected, 3-dimensional combinatorial zeolite is globally rigid in $\R^3$ 
was posed in \cite{SST}. However, \cite[Figure 2]{Jzeo} shows that it is false. An example of a 6-connected 3-dimensional combinatorial zeolite is given in Figure \ref{fig:torus}. It is edge transitive and hence generically rigid in $\R^3$ by Theorem \ref{thm:K4cyclic}.\medskip

\begin{figure}[!t]
        \begin{center}
            \includegraphics[scale=0.3]{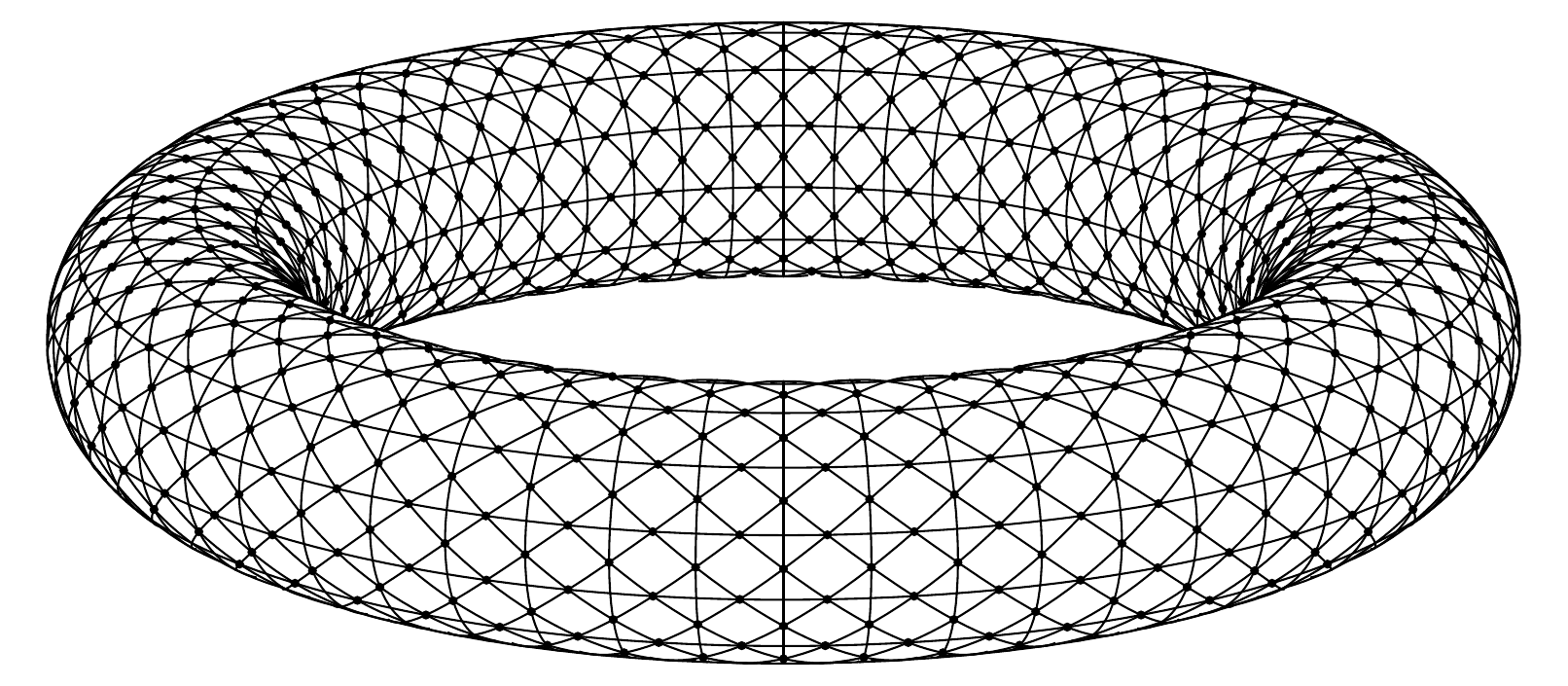}
            \caption{The line graph of a 13 $\times$ 56 toroidal grid. This graph is an edge-transitive 3-dimensional combinatorial zeolite.}\label{fig:torus}
        \end{center}
\end{figure}

We close the paper by asking whether the partition formula in Theorem \ref{thm:BP}
can be converted to a deterministic polynomial algorithm for computing 
 the rank of a body-pin graph in the $\cC_2^1$-cofactor matroid.
 The same question has been posed for the NP$\cap$CoNP characterisation of the rank of an arbitrary graph in the $\cC_2^1$-cofactor matroid given in \cite{CJT} but the special case for body-pin graphs may be more tractable.
 We remark that an efficient randomized algorithm is available by computing the cofactor matrix at a random realisation and evaluating its rank.

\section*{Acknowledgements}

This work was supported by the National Research, Development and Innovation Office
of Hungary, grant no. Advanced 152786, the MTA Distinguished Guest Scientist Fellowship Programme 2025,
and the MTA-ELTE Momentum Matroid Optimization Research Group.

\end{document}